\documentclass[reqno, 12pt]{amsart}
\usepackage{amsmath,amsthm,amsfonts,amssymb}

\date{}
\newtheorem{theorem}{Theorem}[section]
\newtheorem{lemma}{Lemma}[section]
\newtheorem{corollary}{Corollary}[section]
\newtheorem{remark}{Remark}[section]

\newtheorem{example}{Example}[section]

\begin{document}

\centerline{\sc Analysis of mean field games via Fokker-Planck-Kolmogorov equations:}

\centerline{\sc existence of equilibria}

\vskip 2 ex

\begin{center}
{\sc Stanislav V. Shaposhnikov ${}^{a, b}$\footnote{Corresponding author
\par
e-mails: starticle@mail.ru (S.V.Shaposhnikov), shatiltop@mail.ru (D.V.Shatilovich).}, \,
Dmitry V. Shatilovich${}^{a}$}
\end{center}

\vskip 1. ex

\quad ${}^{a}$ Faculty of Mechanics and Mathematics, Lomonosov Moscow State
University, 119991, GSP-1, Leninskie Gory, Moscow, Russia;

\quad ${}^{b}$  National Research University Higher School of Economics, Moscow, Russia.

\vskip 4 ex

{\bf Abstract} We study mean field games with unbounded coefficients. The existence of a solution is proved.
We propose a new approach based on Fokker--Planck--Kolmogorov equations, the Ambrosio--Figalli--Trevisan superposition principle,
the method of doubling variables and a~priory estimates with Lyapunov functions.

\vskip 2 ex


\vskip 2 ex


\vskip 2 ex
	
	\section{\sc Introduction}
	
	\vspace*{0.2cm}

We consider the following nonlinear problem {\bf (P)}: for a given number $T>0$, a probability measure $\nu$,
a control set $U\subset\mathbb{R}^{d_1}$ and functions $a^{ij}$, $b^i$, $q^{im}$, $f$, $g$ construct
a continuous curve $t\mapsto\mu_t$, $t\in [0, T]$, in the space of probability measures on $\mathbb{R}^d$ with the weak topology
and a Borel mapping $u\colon \mathbb{R}^d\times[0, T]\to U$ such that

1) the measure $\mu=\mu_t\,dt$ is a solution to the Cauchy problem for the nonlinear Fokker--Planck--Kolmogorov equation
$$
\partial_t\mu_t=\partial_{x_i}\partial_{x_j}\bigl(a^{ij}(x, t, \mu)\mu_t\bigr)-
\partial_{x_i}\Bigl(\bigl(b^i(x, t, \mu)+q^{im}(x, t, \mu)u_m(x, t)\bigr)\mu_t\Bigr), \quad \mu_0=\nu,
$$

2) the inequality
\begin{multline*}
	\int_0^T\int_{\mathbb{R}^d}f(u(x, t), x, t, \mu)\mu_t(dx)\,dt+\int_{\mathbb{R}^d}g(x, \mu)\mu_T(dx)\le
	\\
	\le\int_0^T\int_{\mathbb{R}^d}f(v(x, t), x, t, \mu)\sigma_t(dx)\,dt+\int_{\mathbb{R}^d}g(x, \mu)\sigma_T(dx)
\end{multline*}
is fulfilled for every Borel mapping $v\colon \mathbb{R}^d\times[0, T]\to U$ and every measure $\sigma=\sigma_t\,dt$ such that
the mapping $t\mapsto\sigma_t$ is a continuous curve in the space of probability measures on $\mathbb{R}^d$ with the weak topology
and $\sigma$ is a solution to the Cauchy problem for the linear Fokker--Planck--Kolmogorov equation
$$
\partial_t\sigma_t=\partial_{x_i}\partial_{x_j}\bigl(a^{ij}(x, t, \mu)\sigma_t\bigr)-
\partial_{x_i}\Bigl(\bigl(b^i(x, t, \mu)+q^{im}(x, t, \mu)v_m(x, t)\bigr)\sigma_t\Bigr), \quad \sigma_0=\nu.
$$
Here the usual convention about summation over repeated indices is employed.

The main result to be presented in Theorem~\ref{th1} is the existence of a solution to the problem~{\bf (P)}.
This problem arises in stochastic mean field games which have the following structure.
Let us fix a measure $\sigma=\sigma_t\,dt$ such that the mapping $t\mapsto\sigma_t$ with values in the space of
probability measures on $\mathbb{R}^d$ is continuous with respect to the weak topology.
Solve the optimal control problem
$$
\inf_{(X_t, U_t)}\mathbb{E}\Bigl(\int_0^Tf(U_t, X_t, t, \sigma)\,dt+g(X_T, \sigma)\Bigr),
$$
where
$$
dX_t=\sqrt{2A(X_t, t, \sigma)}dW_t+\Bigl(b(X_t, t, \sigma)+Q(X_t, t, \sigma)U_t\Bigr)\,dt,
$$
$$
{\rm Law}\bigl(X_0\bigr)=\nu, \, Q(x, t, \sigma)=\bigl(q^{im}(x, t, \sigma)\bigr), \,
A(x, t, \sigma)=\bigl(a^{ij}(x, t, \sigma)\bigr), \,
b(x, t, \sigma)=\bigl(b^i(x, t, \sigma)\bigr).
$$
Denote by $\Phi(\sigma)$ the set of all measures $\eta=\eta_t\,dt$
such that there exists a solution $(X_t, U_t)$ of this optimal control
problem and $\eta_t={\rm Law}\bigl(X_t\bigr)$ for all $t\in[0, T]$.
A mean field game solution $\mu$ is a fixed point
of the mapping $\sigma\mapsto \Phi(\sigma)$, that is $\mu\in \Phi(\mu)$.
Mean field games describe the Nash equilibria in games with an infinite number of agents.
In this interpretation the solution~$\mu$ represents the distribution of state processes an infinity of agents and the corresponding $X_t$ describes
the state process dynamics of a single representative agent.
Mean field games are used as approximations to the Nash equilibria in stochastic games with many players.
The solution $(u, \mu)$ of the problem~{\bf (P)} and the Ambrosio--Figally--Trevisan superposition principle
together with the convexity assumptions allow us to solve the mean field game.
This will be discussed in Corollary~{\ref{cor1}.
	Note that if $\mu_t={\rm Law}\bigl(X_t\bigr)$ and $U_t=u(X_t, t)$ for some Borel function~$u$, then
	$$
	\mathbb{E}\Bigl(\int_0^Tf(U_t, X_t, t, \sigma)\,dt+g(X_T, \sigma)\Bigr)=
	\int_0^T\int_{\mathbb{R}^d}f(u(x, t), x, t, \sigma)\mu_t(dx)\,dt+\int_{\mathbb{R}^d}g(x, \sigma)\mu_T(dx)
	$$
	and by the It$\hat{o}$ formula the measure $\mu=\mu_t\,dt$ is a solution to the
	Fokker--Planck--Kolmogorov equation. According to the superposition principle, under broad assumptions
	every probability solution to the Fokker--Planck--Kolmogorov equation
	can be represented by a weak solution to the corresponding stochastic differential equation.
	Thus the problem {\bf (P)} is a special case (when $U_t=u(X_t, t)$) of the mean field game.
	Moreover, under the convexity assumptions the mean field game can be reduced to the problem~{\bf (P)}.
	
	The study of mean field games began with the pioneering works of Huang, Malham$\grave{e}$ and Caines \cite{HMC06}
	and Lasry and Lions \cite{Lasry-Lions}. A survey of known results is given in the books
	\cite{CardParetta20}, \cite{BensoussanFrehseYam}, \cite{CarmonaDel} and \cite{Gomes}.
	Note also that mean field games were largely developed in Pierre-Louis Lions's
	series of lectures at the Coll$\grave{e}$ge de France. The description of these lectures is presented in~\cite{CardParetta20}. The study of mean field games has been motivated, at least in part, by the increasing number of applications. For example, the application of mean field games to optimal liquidation problems can be found in \cite{Cardaliaguet18}. The paper \cite{Carmona20} provides a detailed overview of other applications.
	
	There are several ways to construct solutions to mean field games. The first approach is to solve
	the forward-backward problem which has the form of a system of Hamilton--Jacobi--Bellman and
	Fokker--Planck--Kolmogorov equations (see, for instance \cite{Barbu25}, \cite{Paretta17}).
	In this case coefficients have a special structure, for instance, the diffusion matrix $A$ is non-degenerate and even constant
	and $b=-H_p(x, t, p)$, in particular $A$ and $b$ do not depend on $\mu$. Moreover, the function~$f$ has the form
	$f(u, x, t, \mu)=l(u, x, t)+F(x, t, \mu)$. However,
	this approach allows to take into account the nonlinearity of local type in $\mu$
	and non-smooth coefficients.
	The second approach is based on a probabilistic analysis of mean field games (see \cite{CarmonaDelarue},
	\cite{Carmona-Lacker}, \cite{Lacker}). The probabilistic approach is developed in three directions:
	1)~the~stochastic maximum principle (see \cite{BensoussanFYpaper}, \cite{CarmonaDelarue}),
	2)~the~convergence of the Nash equilibria in a symmetric $N$--player game to the mean field game limit
	(see \cite{Cardconv}, \cite{Fischer converg}, \cite{Lacker converg}),
	3)~relaxed controls (see \cite{Carmona-Lacker}, \cite{Lacker}).
	The stochastic maximum principle and the convergence of Nash equilibria require restrictive smoothness conditions on the coefficients.
	More general results use relaxed controls. The idea is as follows: first, we construct a measure on the control set $U$
	and then, using the convexity of the data and conditional measures, we obtain a control function $(x, t)\mapsto u(x, t)$.
	This idea is well--known in optimal control problems (see \cite{Filippov}).
	
	One of the most general existence results is presented in~\cite{Lacker},
	where there are two types of assumptions: (A) the coefficients are Lipschitzian and
	(C) the coefficients are bounded and continuous, the matrix $A$ is non-degenerate and the control set $U$ is compact.
	Note that the assumptions of the paper~\cite{Lacker} only allow  linear growth of the coefficients.     	
	In addition, if $U=\mathbb{R}^{d_1}$, then the assumption (A.2) from \cite{Lacker} implies that $Q(x, t, \mu)$ does not depend on~$x$.
	We improve the results of \cite{Lacker} and allow nonlinear growth of the coefficients and
	more general dependence on $\mu$ (see examples \ref{ex1}--\ref{ex4}).
	In \cite{Lacker} the relaxed control is a random process with values in the space of probability measures, the theory of martingale measures is applied. This approach is generalised with similar assumptions
	to mean field games with singular controls \cite{Fu-Horst} and to mean field games with absorption \cite{Campi}.
	We propose a different approach based on Fokker--Planck--Kolmogorov equations and not involving stochastic control theory.
	Moreover, a~priori estimates with Lyapunov functions allow us to consider coefficients with arbitrary growth.
	Another improvement is that the coefficients $A$, $b$, $Q$ and the functions $f$, $g$ can depend on
	the entire measure $\mu=\mu_t\,dt$ rather than on the measure~$\mu_t$ at time~$t$.
	The connection between probabilistic solutions of Fokker--Planck--Kolmogorov equations
	and solutions to stochastic differential equations is based on the Ambrosio--Figalli--Trevisan superposition principle,
	which is now known under very general assumptions (see \cite{Superp21}, \cite{Trev}).
    Moreover, we use the method of doubling variables (see \cite{LebRi}, \cite{Chen}).
	References to the latest results concerning the theory of Fokker--Planck--Kolmogorov equations can be found in \cite{bookFPK}.
	
	Note that if the coefficients $A$, $b$, $Q$ and the functions $f$, $g$ do not depend on $\mu$,
	then the problem~{\bf (P)} is a control problem for linear Fokker--Planck--Kolmogorov equations
	(see, for instance \cite{Anita1}, \cite{Anita2}, \cite{AnuBor}). If $f=g=0$ and $Q=0$, then the problem~{\bf (P)}
	is a Cauchy problem for the nonlinear Fokker--Planck--Kolmogorov equation
	(see, for instance  \cite{BarbuRochner}, \cite{BogachShap}, \cite{Frank}, \cite{kolokoltsov-nonl}).
	We stress that the problem~{\bf (P)} is not a control problem for nonlinear
	Fokker--Planck--Kolmogorov equations (see \cite{Carrillo20}). The case $A=0$ corresponds to the deterministic mean field games which are discussed
	in papers \cite{Averboukh first order1}, \cite{Cannarsa}, \cite{Cardaliaguet-Graber}.
	Finally, note that mean field games for more general Markov processes are considered
	in works \cite{Averboukh converg Levy}, \cite{Chowdhury-Espen-Krupski}, \cite{Kolokoltsov}.
	
	This paper consists of four sections. In Section 2 we discuss the main results and examples.
	Auxiliary results are proved in Section 3. Section 4 is devoted to the proofs of the main results.

	\vspace*{0.2cm}
	
	\section{\sc Main results}
	
	\vspace*{0.2cm}
	
	Let $T>0$, $V\in C^2(\mathbb{R}^d)$, $V\ge 0$ and $\lim_{|x|\to+\infty}V(x)=+\infty$.
	For example, one can take the function $V(x)=(1+|x|^2)^{p/2}$, where $p>0$.
    The bounded Borel measure $\mu$ on $\mathbb{R}^d$ is called a probability measure if $\mu\ge 0$ and $\mu(\mathbb{R}^d)=1$.
    We say that the bounded Borel measure $\mu$ on $\mathbb{R}^d\times[0, T]$ is given by a family of Borel measures $(\mu_t)_{t\in[0, T]}$
    if for every Borel set $E$ the mapping $t\mapsto \mu_t(E)$ is Borel measurable and for every bounded Borel function~$\eta$ the equality
    $$
	\int_{\mathbb{R}^d\times[0, T]}\eta(x, t)\mu(dxdt)=\int_0^T\int_{\mathbb{R}^d}\eta(x, t)\mu_t(dx)\,dt
	$$
    holds. We also use the short notation $\mu=\mu_t\,dt$.

    Let $\mathcal{M}(V)$ denote the set of bounded nonnegative Borel measures $\mu$ on
	$\mathbb{R}^d\times [0, T]$ given by a family of probability measures
    $(\mu_t)_{t\in[0, T]}$ such that the mapping $t\mapsto\mu_t$ is continuous with respect to the weak topology
    (see Remark~\ref{r0}) and
	$$
	\sup_{t\in[0, T]}\int_{\mathbb{R}^d}V(x)\,\mu_t(dx)<\infty.
	$$
	We will say that measures $\mu^n=\mu^n_t\,dt$ in the set $\mathcal{M}(V)$
	converge $V$ -- weakly to a measure $\mu=\mu_t\,dt$ in $\mathcal{M}(V)$
    if for all $t\in [0, T]$ one has the equality
	$$
	\lim_{n\to\infty}\int_{\mathbb{R}^d}\zeta(x)\mu_t^n(dx)=\int_{\mathbb{R}^d}\zeta(x)\mu_t(dx)
	$$
	for every continuous function $\zeta$ on $\mathbb{R}^d$ with
	$\lim_{|x|\to+\infty}\zeta(x)/V(x)=0$ (see also Remark~\ref{r2}).

	For $R>0$ and $M\ge 0$ let $\mathcal{M}_{R, M}(V)$ denote the set of all measures $\mu=\mu_t\,dt$
    in $\mathcal{M}(V)$ such that for every $t\in[0, T]$ one has the estimate
	$$
	\int_{\mathbb{R}^d}V(x)\,\mu_t(dx)\le Re^{Mt}.
	$$
    The set $\mathcal{M}_{R, 0}(V)$ is denoted by $\mathcal{M}_R(V)$.

	For every measure $\mu\in \mathcal{M}(V)$
	for every $1\le i, j\le d$ and $1\le m\le d_1$,
	we are given Borel functions
	$$
	(x, t)\mapsto a^{ij}(x, t, \mu), \quad (x, t)\mapsto b^i(x, t, \mu), \quad (x, t)\mapsto q^{im}(x, t, \mu)
	$$
	such that the matrix $A(x, t, \mu)=\bigl(a^{ij}(x, t, \mu)\bigr)_{1\le i, j\le d}$ is symmetric and nonnegative definite.
	
	Let $L_{\mu}$ denote the differential operator
	$$
	L_{\mu}\psi(x, t)={\rm tr}(A(x, t, \mu)D^2\psi(x))+\langle b(x, t, \mu), \nabla\psi(x)\rangle,
	$$
	where $b(x, t, \mu)=(b^i(x, t, \mu))_{1\le i\le d}$.
	
	Let $U$ be a nonempty convex closed set in $\mathbb{R}^{d_1}$.
	For every $u\in U$ let $L_{\mu, u}$ denote the differential operator
	$$
	L_{\mu, u}\psi(x, t)=L_{\mu}\psi(x, t)+\langle Q(x, t, \mu)u, \nabla\psi(x)\rangle,
	$$
	where $Q(x, t, \mu)=\bigl(q^{im}(x, t, \mu)\bigr)_{1\le i\le d, 1\le m\le d_1}$.	
    The transpose of the matrix $Q$ is denoted by $Q^{\top}$.
	
	Let $W$ be a continuous function on $\mathbb{R}^d$ such that
	$$
	0\le W(x)\le V(x), \quad \lim_{|x|\to+\infty}\frac{W(x)}{V(x)}=0.
	$$
	
	We also use the convex increasing and continuous function $h$ on $[0, +\infty)$ such that
	$$
	h(0)=0, \quad \lim_{v\to+\infty}\frac{h(v)}{v}=+\infty.
	$$
	
	Let $h^{*}$ denote the Legendre transform of the function $h$ that is
	$$h^{*}(v)=\sup_{p\ge 0}\bigl(pv-h(p)\bigr).$$
    Below we often use the inequality $pv\le h(p)+h^{*}(v)$. Note also that the function $h^{*}$ plays a significant
    role in the growth conditions on our coefficients $A$, $b$ and $Q$ while the function $h$ limits the
    growth of the function $f$ with respect to the control $u$.

    Let us formulate our main assumptions.
	
	\vspace*{0.2cm}
	
	$\bf(H1)$ {\bf (Local conditions)}
	
	\vspace*{0.2cm}
	
	$\bf(H1.1)$
	For every open ball $B\subset\mathbb{R}^d$, for every $R>0$ and all $i, j, m$ there holds
	$$
	\sup_{x\in B, t\in[0, T], \mu\in\mathcal{M}_R(V)}
	\Bigl(|a^{ij}(x, t, \mu)|+|b^i(x, t, \mu)|+|q^{im}(x, t, \mu)|\Bigr)<\infty.
	$$
	
	$\bf(H1.2)$
	For every measure $\mu\in\mathcal{M}(V)$ and all $t\in[0, T]$ the functions
	$$
	x\mapsto a^{ij}(x, t, \mu), \quad x\mapsto b^{i}(x, t, \mu), \quad x\mapsto q^{im}(x, t, \mu)
	$$
	are continuous on $\mathbb{R}^d$.
	
	$\bf(H1.3)$
	For every open ball $B$, for all $t\in[0, T]$ and for every number $R>0$ the $V$--weak convergence of measures
    $\mu^n\in \mathcal{M}_{R}(V)$ to a measure $\mu\in \mathcal{M}_{R}(V)$ implies the equality
	\begin{multline*}
		\lim_{n\to\infty}\sup_{x\in B}\Bigl(|a^{ij}(x, t, \mu^n)-a^{ij}(x, t, \mu)|+
		\\
		|b^{i}(x, t, \mu^n)-b^{i}(x, t, \mu)|
		+|q^{im}(x, t, \mu^n)-q^{im}(x, t, \mu)|\Bigr)=0.
	\end{multline*}
	
	We also need some global assumptions with the function~$V$.
	
	\vspace*{0.2cm}
	
	$\bf(H2)$ {\bf (Global conditions)}
	
	\vspace*{0.2cm}
	
	$\bf(H2.1)$ There exists a number $C_L>0$ such that the estimate
	$$
	L_{\mu}V(x, t)+h^{*}\bigl(|Q^{\top}(x, t, \mu)\nabla V(x)|\bigr)
	\le C_LV(x)+C_L\int_{\mathbb{R}^d}V(y)\mu_t(dy)+C_L\sup_{t\in[0, T]}\int_{\mathbb{R}^d}W(y)\mu_t(dy)
	$$
	holds for every measure $\mu\in \mathcal{M}(V)$, every $x\in\mathbb{R}^d$ and all $t\in[0, T]$.
	
	$\bf(H2.2)$
	For every measure $\mu\in \mathcal{M}(V)$ there exist a number $C_1(\mu)>0$
    and a nonnegative Borel function $\Theta$ on $\mathbb{R}^d\times[0, T]$ such that
	for all $x, y\in\mathbb{R}^d$ and $t\in[0, T]$ one has
	\begin{multline*}
	{\rm trace}\Biggl(\Bigl(\sqrt{A(x, t, \mu)}-\sqrt{A(y, t, \mu)}\Bigr)^2\Biggr)+\Big\langle b(x, t, \mu)-b(y, t, \mu), x-y\Big\rangle\le
\\
	\le C_1(\mu)\Bigl(1+V(x)+V(y)\Bigr)|x-y|^2
	\end{multline*}
	and
	$$
	\|Q(x, t, \mu)-Q(y, t, \mu)\|\le \Bigl(\Theta(x, t)+\Theta(y, t)\Bigr)|x-y|.
	$$
	
	$\bf(H2.3)$ For every measure $\mu\in \mathcal{M}(V)$ there exists a number $C_2(\mu)>0$ such that
	for all $x\in\mathbb{R}^d$ and $t\in[0, T]$ one has the inequality
	$$
	\|A(x, t, \mu)\|+|b(x, t, \mu)|+h^{*}(\|Q(x, t, \mu)\|)+h^{*}(\Theta(x, t))\le C_2(\mu)V(x).
	$$
	
	Suppose that for every measure $\mu\in\mathcal{M}(V)$
	we are given Borel functions
	$$
	(u, x, t)\mapsto f(u, x, t, \mu) \quad \hbox{\rm and} \quad x\mapsto g(x, \mu)
	$$
    on $U\times \mathbb{R}^d\times[0, T]$ and on $\mathbb{R}^d$, respectively.
	
	\vspace*{0.2cm}
	
	$\bf(H3)$ {\bf (Conditions on $f$ and $g$)}
	
	\vspace*{0.2cm}
	
	$\bf(H3.1)$ The function $g$ is continuous in $x$ and there exists a number $C_g>0$
such that for all $x\in\mathbb{R}^d$ and every $\mu\in\mathcal{M}(V)$ we have
	$$
	|g(x, \mu)|\le C_g\Bigl(W(x)+\sup_{t\in[0, T]}\int_{\mathbb{R}^d}W(y)\mu_t(dy)\Bigr).
	$$
	
	$\bf(H3.2)$ There exist numbers $C_{h}>1$ and $C_f>0$ such that the inequalities
	\begin{multline*}
		h(|u|)-C_f\Bigl(W(x)+\sup_{t\in[0, T]}\int_{\mathbb{R}^d}W(y)\mu_t(dy)\Bigr)
		\le f(u, x, t, \mu)\le
		\\
		\le C_{h}h(|u|)+C_f\Bigl(W(x)+\sup_{t\in[0, T]}\int_{\mathbb{R}^d}W(y)\mu_t(dy)\Bigr)
	\end{multline*}
	hold for every $\mu\in \mathcal{M}(V)$, $u\in\mathbb{R}^{d_1}$, $x\in\mathbb{R}^d$ and $t\in[0, T]$.
	
	$\bf(H3.3)$ For all open balls $B\subset\mathbb{R}^d$, $B_1\subset\mathbb{R}^{d_1}$
	and every number $R>0$ the $V$--weak convergence of measures $\mu^n\in \mathcal{M}_{R}(V)$
	to a measure $\mu\in \mathcal{M}_{R}(V)$ implies the equality
	$$
	\lim_{n\to\infty}\sup_{x\in B}|g(x, \mu^n)-g(x, \mu)|=0
	$$
	and for every $t\in [0, T]$ the equality
	$$
	\lim_{n\to\infty}\sup_{u\in B_1\cap U, x\in B}|f(u, x, t, \mu^n)-f(u, x, t, \mu)|=0.
	$$
	
	$\bf(H3.4)$  The function $f$ is convex in $u$ and the mapping
	$$
	(u, x)\mapsto f(u, x, t, \mu)
	$$
	is continuous on the set $U\times\mathbb{R}^d$ for every measure $\mu\in \mathcal{M}(V)$ and all $t\in[0, T]$.

	Our main result is the following theorem.
	
	\begin{theorem}\label{th1}
		Assume that the conditions $\rm (H1), (H2), (H3)$ are fulfilled and $\nu$ is a probability measure on
        $\mathbb{R}^d$ with $V\in L^1(\nu)$. Then there exists a mapping
		$t\mapsto\mu_t$ from $[0, T]$ to the space of probability measures on $\mathbb{R}^d$ that is continuous
        with respect to the weak topology and there exists a Borel function $(x, t)\mapsto u(x, t)$ from $\mathbb{R}^d\times[0, T]$ to $U$
		such that
		
		{\rm (i)} the measure $\mu=\mu_t\,dt$ belongs to $\mathcal{M}(V)$ and the function $(x, t)\mapsto h(|u(x, t)|)$
		is integrable with respect to the measure $\mu$ on $\mathbb{R}^d\times[0, T]$;
		
		{\rm (ii)} the measure $\mu=\mu_t\,dt$ is a solution to the Cauchy problem for the Fokker--Planck--Kolmogorov equation
		$$
		\partial_t\mu_t=L_{\mu, u(x, t)}^{*}\mu_t, \quad \mu_0=\nu,
		$$
		that is for every $\psi\in C_0^{\infty}(\mathbb{R}^d)$ and for all $t\in[0, T]$ we have
		$$
		\int_{\mathbb{R}^d}\psi(x)\mu_t(dx)-\int_{\mathbb{R}^d}\psi(x)\nu(dx)=\int_0^{t}\int_{\mathbb{R}^d}L_{\mu, u(x, s)}\psi(x, s)\,\mu_s(dx)\,ds;
		$$
		
		{\rm (iii)}
		the inequality
		\begin{multline*}
			\int_0^{T}\int_{\mathbb{R}^d}f(u(x, t), x, t, \mu)\,\mu_t(dx)\,dt
			+\int_{\mathbb{R}^d}g(x, \mu)\mu_{T}(dx)\le
			\\
			\le\int_0^{T}\int_{\mathbb{R}^d}f(v(x, t), x, t, \mu)\,\sigma_t(dx)\,dt
			+\int_{\mathbb{R}^d}g(x, \mu)\sigma_{T}(dx)
		\end{multline*}
		holds for every measure $\sigma_t\,dt$ given by a continuous curve $t\mapsto\sigma_t$
        in the space of probability measures with the weak topology and for every Borel function $(x, t)\mapsto v(x, t)$
		from $\mathbb{R}^d\times[0, T]$ to $U$ such that the function $(x, t)\mapsto h(|v(x, t)|)$ is integrable with respect to the measure
		$\sigma_t\,dt$ and the measure $\sigma_t\,dt$ satisfies the Cauchy problem
		$\partial_t\sigma_t=L_{\mu, v(x, t)}^{*}\sigma_t$, $\sigma_0=\nu$.
	\end{theorem}

Note that in Theorem~\ref{th1} we do not assume that $\sigma\in\mathcal{M}(V)$
since this assumption follows from
the integrability of the function $(x, t)\mapsto h(|v(x, t)|)$ (see Remark~\ref{rem-int}).
For measurable spaces $(X, \mathcal{X})$, $(Y, \mathcal{Y})$, a measure $P$ on $\mathcal{X}$
and a measurable mapping $\phi\colon X\to Y$ let $P\circ\phi^{-1}$ denote the push--forward measure,
that is $P\circ\phi^{-1}(E)=P\bigl(\phi^{-1}(E)\bigr)$ for every $E\in\mathcal{Y}$.

Applying Theorem~\ref{th1} and the Ambrosio--Figalli--Trevisan superposition principle,
we derive the following statement for stochastic mean field games.

	\begin{corollary}\label{cor1}
		Assume that the conditions $\rm (H1), (H2), (H3)$ are fulfilled and $\nu$ is a probability measure
        on $\mathbb{R}^d$ with $V\in L^1(\nu)$. Then there exists a mapping
		$t\mapsto\mu_t$ from $[0, T]$ to the space of probability measures on $\mathbb{R}^d$ that is continuous
        with respect to the weak topology and there exists a Borel function $(x, t)\mapsto u(x, t)$ from $\mathbb{R}^d\times[0, T]$ to $U$
		such that
		
		{\rm (i)} $\mu_0=\nu$, the measure $\mu=\mu_t\,dt$ belongs to $\mathcal{M}(V)$ and the function $(x, t)\mapsto h(|u(x, t)|)$
		is integrable with respect to the measure $\mu$ on $\mathbb{R}^d\times[0, T]$,
		
		{\rm (ii)} there exists a filtered probability space $(\Omega, \mathcal{F}_t, \mathcal{P})$ supporting
		a~$\mathcal{F}_t$--Brownian motion $W$ and a $\mathcal{F}_t$--adapted process $X$
      	such that
		$$
		dX_t=\sqrt{2A(X_t, t, \mu)}dW_t+\bigl(b(X_t, t, \mu\bigr)+Q(X_t, t, \mu)u(X_t, t)\bigr)\,dt
		$$
		and $\mathcal{P}\circ X_t^{-1}=\mu_t$ for all $t\in[0, T]$,
		
		{\rm (iii)} if $(\widetilde{\Omega}, \widetilde{\mathcal{F}}_t, \widetilde{\mathcal{P}})$
        is another filtered probability space
        supporting a $\widetilde{\mathcal{F}}_t$--Brownian motion $\widetilde{W}$,
        a~$\widetilde{\mathcal{F}}_t$--adapted process $Y$ and a $\widetilde{\mathcal{F}}_t$--adapted process $V$
      	such that
		$$
		dY_t=\sqrt{2A(Y_t, t, \mu)}d\widetilde{W}_t+\bigl(b(Y_t, t, \mu\bigr)+Q(Y_t, t, \mu)V_t\bigr)\,dt, \quad
        \nu=\widetilde{\mathcal{P}}\circ Y_0^{-1}, \quad \mathbb{E}h(|V_t|)<\infty,
		$$
		then
		$$
		\mathbb{E}\Bigl[\int_0^{T}f(u(X_t, t), X_t, t, \mu)\,dt+g(X_{T}, \mu)\Bigr]\le
		\mathbb{E}\Bigl[\int_0^{T}f(V_t, Y_t, t, \mu)\,dt+g(Y_{T}, \mu)\Bigr].
		$$
        \end{corollary}

This stochastic differential equation is commonly referred to as the granular media equation
or the McKean--Vlasov equation. In the last two decades this equation and the corresponding nonlinear
Fokker--Planck--Kolmogorov equation have been being studied intensively.
The results on the existence and uniqueness of the solution
can be found in \cite{BogachShap}, \cite{Veret20}, \cite{Huang}.
We stress that nonlinear Fokker--Planck--Kolmogorov equations with coefficients rapidly increasing at infinity
appear in many physical models with unbounded potentials (see \cite{Frank}).
Moreover, in socio-economics models and in traffic models linear and nonlinear
Fokker–Planck–Kolmogorov equations are often considered on a half-line or on an interval with
diffusion coefficients degenerating at the boundary. By a change of variables such
equations are transformed to equations on the line with coefficients rapidly increasing at infinity.
Theorem~\ref{th1} and Corollary~\ref{cor1} allow us to consider mean field games with coefficients rapidly increasing at infinity.

Let us consider examples illustrating the conditions (H1.1), (H2.1), (H2.2), (H2.3), (H3.1) and (H3.2).

\begin{example}\label{ex1}\rm
Let $V(x)=1+|x|^2$, $W(x)=|x|$ and $h(v)=Cv^2$, where $C>0$.
Suppose that there exist positive numbers $C_1$, $C_2$, $C_3$, $C_4$, $C_5$ such that
for every $x, y\in\mathbb{R}^d$, all $t\in[0, T]$ and every $\mu\in\mathcal{M}(V)$
we have

(i) $\|\sqrt{A(x, t, \mu)}-\sqrt{A(y, t, \mu)}\|+\|Q(x, t, \mu)-Q(y, t, \mu)\|+|b(x, t, \mu)-b(y, t, \mu)|\le C_1|x-y|$,

(ii) $\sup_{y\in\mathbb{R}^d}\|Q(y, t, \mu)\|\le C_2$ and
$\displaystyle\|\sqrt{A(0, t, \mu)}\|+|b(0, t, \mu)|\le C_2+C_2\int_{\mathbb{R}^d}|y|\mu_t(dy)$,

(iii) $\displaystyle|g(x, \mu)|\le C_3\bigl(1+|x|\bigr)+C_3\int_{\mathbb{R}^d}|y|\mu_t(dy)$,

(iv) $\displaystyle C|u|^2-C_4\bigl(1+|x|\bigr)-C_4\int_{\mathbb{R}^d}|y|\mu_t(dy)\le f(u, x, t, \mu)\le$

\rightline{$\displaystyle \le C_5|u|^2+C_4\bigl(1+|x|\bigr)+C_4\int_{\mathbb{R}^d}|y|\mu_t(dy)$.}

Then the conditions (H1.1), (H2.1), (H2.2), (H2.3), (H3.1) and (H3.2) are fulfilled.

\begin{proof}
According to the conditions (i) and (ii), we have
$$
\|\sqrt{A(x, t, \mu)}\|+|b(x, t, \mu)|\le C_2+C_1|x|+C_2\int_{\mathbb{R}^d}|y|\mu_t(dy)
$$
Note that
$$
L_{\mu}V(x, t)+h^{*}(|Q^{\top}(x, t)\nabla V(x)|)=
2{\rm trace}\, A(x, t, \mu)+2\langle b(x, t, \mu), x\rangle+\frac{1}{C}|Q^{\top}(x, t)x|^2.
$$
Applying the Cauchy inequality we obtain the estimates
$$
2\big|\langle b(x, t, \mu), x\rangle\big|\le |b(x, t, \mu)|^2+|x|^2, \quad
\Bigl(\int_{\mathbb{R}^d}|x|\mu_t(dx)\Bigr)^2\le \int_{\mathbb{R}^d}|x|^2\mu_t(dx).
$$
which imply
$$
L_{\mu}V(x, t)+h^{*}(|Q^{\top}(x, t)\nabla V(x)|)\le N+NV(x)+\int_{\mathbb{R}^d}V(x)\mu_t(dx),
$$
where the constant $N$ does not depend on $x$, $t$ and $\mu$.
\end{proof}

In the same way we can consider $V(x)=(1+|x|^2)^{s/2}$, $W(x)=|x|^p$ and $h(v)=Cv^r$,
where we assume that $1\le p<s\le r$. In this case we obtain the same conditions but in the condition (ii)
the integral $\displaystyle\int_{\mathbb{R}^d}|x|\mu_t(dx)$ is replaced by the integral
$\displaystyle\Bigl(\int_{\mathbb{R}^d}|x|^p\mu_t(dx)\Bigr)^{1/p}$, in the conditions (iii) and (iv)
the integral $\displaystyle\int_{\mathbb{R}^d}|x|\mu_t(dx)$ is replaced by the integral
$\displaystyle\int_{\mathbb{R}^d}|x|^p\mu_t(dx)$ and $|u|^2$ is replaced by $|u|^r$.

If in addition to the conditions (i), (ii), (iii), (iv)
the continuity conditions (H1.2), (H1.3), (H3.3) and (H3.4) hold, then the conditions (H1), (H2) and (H3) are fulfilled.
\end{example}

\begin{example}\label{ex2}\rm
Let $V(x)=1+|x|^m$, $W(x)=|x|^p$ and $h(v)=Cv^2$, where $m\ge 2$, $1\le p<m$ and $C>0$.
Suppose that there exist number $\varepsilon\in(0, 1/2)$ and positive numbers $C_1$, $C_2$, $C_3$, $C_4$, $C_5$, $C_6$
such that for every $x, y\in\mathbb{R}^d$, all $t\in[0, T]$ and every $\mu\in\mathcal{M}(V)$
we have

(i) $\|\sqrt{A(x, t, \mu)}-\sqrt{A(y, t, \mu)}\|+\|Q(x, t, \mu)-Q(y, t, \mu)\|\le C_1|x-y|$ and
$$
\|\sqrt{A(x, t, \mu)}\|\le C_1\bigl(1+|x|\bigr), \quad \|Q(x, t, \mu)\|\le C_1\bigl(1+|x|^{\varepsilon}\bigr),
$$

(ii) $\displaystyle\langle b(x, t, \mu), x\rangle\le C_2-C_3|x|^{m+1}$ and
$$|b(x, t, \mu)-b(y, t, \mu)|\le C_2(1+|x|^{m-1}+|y|^{m-1})|x-y|,$$

For instance, this condition holds for $b(x, t, \mu)=-x|x|^{m-1}+b_0(t, \mu)$, where $b_0$ is a bounded vector field.

(iii) $\displaystyle|g(x, \mu)|\le C_4\bigl(1+|x|^p\bigr)+C_4\int_{\mathbb{R}^d}|y|^p\mu_t(dy)$,

(iv) $\displaystyle C|u|^2-C_5\bigl(1+|x|^p\bigr)-C_5\int_{\mathbb{R}^d}|y|^p\mu_t(dy)\le f(u, x, t, \mu)\le$

\rightline{$\displaystyle \le C_6|u|^2+C_5\bigl(1+|x|^p\bigr)+C_5\int_{\mathbb{R}^d}|y|^p\mu_t(dy)$.}

Then the conditions (H1.1), (H2.1), (H2.2), (H2.3), (H3.1) and (H3.2) are fulfilled.

\begin{proof}
Using the equality
$$
L_{\mu}V(x)=m|x|^{m-2}{\rm trace} A(x, t, \mu)+m(m-2)|x|^{m-4}\langle A(x, t, \mu)x, x\rangle+
m|x|^{m-2}\langle b(x, t, \mu), x\rangle,
$$
we obtain the estimate
$$
L_{\mu}V(x, t)\le N_1+N_1|x|^m-mC_2|x|^{2m-1},
$$
where the constant $N_1$ does not depend on $x$, $t$ and $\mu$.
Moreover, we have
$$
h^{*}(|Q^{\top}(x, t, \mu)\nabla V(x)|)=\frac{m^2}{4C}\|Q(x, t, \mu)\|^2|x|^{2m-2}\le \frac{m^2C_1^2}{2C}|x|^{2m-2}+
\frac{m^2C_1^2}{2C}|x|^{2m-2+2\varepsilon}.
$$
Since $2m-2+2\varepsilon<2m-1$, we arrive at the estimate
$$
L_{\mu}V(x, t)+h^{*}(|Q^{\top}(x, t, \mu)\nabla V(x)|)\le N_2,
$$
where the constant $N_2$ does not depend on $x$, $t$ and $\mu$. Finally, note that
$$
h^{*}(\|Q(x, t, \mu)\|)\le \frac{C_1^2}{4C}\bigl(1+|x|^{\varepsilon}\bigr)^2\le  N_3V(x),
$$
where the constant $N_3$ also does not depend on $x$, $t$ and $\mu$.
\end{proof}

If in addition to the conditions (i), (ii), (iii), (iv)
the continuity conditions (H1.2), (H1.3), (H3.3) and (H3.4) hold, then the conditions (H1), (H2) and (H3) are fulfilled.
\end{example}

In general, in order to construct functions $h$, $V$ and $W$ the following approach is suggested.
Firstly, we find functions $h$ and $W$ such that the conditions (H3.1) and (H3.2) are fulfilled.
Secondary, using the functions $h$ and $W$, we obtain a function $V$ such that $V\in C^2(\mathbb{R}^d)$,
$W\le V$, $\lim_{|x|\to\infty}V(x)=+\infty$, $\lim_{|x|\to\infty}W(x)/V(x)=0$ and the condition (H2.1) is fulfilled.
Thirdly, we find a function $\Theta$ such that the condition (H2.2) holds.
Finally, we verify the remaining growth conditions.

Let us consider examples illustrating the conditions (H1.2), (H1.3) and (H3.3).

\begin{example}\label{ex3}\rm
Let $p\ge 1$. By $\mathcal{P}_p(\mathbb{R}^d)$ we denote the space
of probability measures $\mu$ such that $|x|^p\in L^1(\mu)$.
Recall that the Kantorovich distance $W_p(\mu, \sigma)$ of order $p$ is defined as the infimum of the integral
$$
\int_{\mathbb{R}^d\times\mathbb{R}^d}|x-y|^p\pi(dxdy)
$$
over all probability measures $\pi$ on $\mathbb{R}^d\times\mathbb{R}^d$ with
projections $\mu$ and $\sigma$ onto the factors (see \cite{BK}). We consider the space $\mathcal{P}_p(\mathbb{R}^d)$
with the distance $W_p$.

Let $\Psi\colon \mathbb{R}^d\times[0, T]\times \mathcal{P}_p(\mathbb{R}^d)\to\mathbb{R}$ be a Borel function.
Assume that for every $t\in[0, T]$ the function $\Psi(x, t, \eta)$ is continuous in $(x, \eta)$.
Let $V(x)=\bigl(1+|x|^2\bigr)^{s/2}$ and $s>p$. For every $\mu\in\mathcal{M}(V)$ we define
$$
\psi(x, t, \mu)=\Psi(x, t, \mu_t).
$$
Let $R>0$. Then the $V$--weak convergence of measures $\mu^n=\mu^n_t\,dt\in\mathcal{M}_R(V)$
to a measure $\mu=\mu_t\,dt\in\mathcal{M}_R(V)$
implies that the equality
$$
\lim_{n\to\infty}\sup_{x\in B}\big|\psi(x, t, \mu^n)-\psi(x, t, \mu)\big|=0
$$
holds for every $t\in[0, T]$ and every open ball $B\subset\mathbb{R}^d$.

\begin{proof}
Let $\overline{B}$ denote the closure of a ball $B$.
Note that $\mathcal{M}_R(V)$ is a compact set in $\mathcal{P}_p(\mathbb{R}^d)$.
It follows that $\overline{B}\times\mathcal{M}_R(V)$ is a compact set and
for every $t\in[0, T]$ the mapping $(x, \eta)\mapsto \Psi(x, t, \eta)$ is uniformly continuous
on $\overline{B}\times\mathcal{M}_R(V)$.
Finally, note
 that the $V$--weak convergence of measures $\mu^n=\mu^n_t\,dt\in\mathcal{M}_R(V)$
to a measure $\mu=\mu_t\,dt\in\mathcal{M}_R(V)$
implies the equality $\lim_{n\to\infty}W_p(\mu_t^n, \mu_t)=0$ for every $t\in[0, T]$.
\end{proof}

Thus if each of the functions $a^{ij}$, $b^i$, $q^{im}$, $g$ and $f$
is given by the same rule as the function~$\psi$, then the conditions
(H1.2), (H1.3), (H3.3) are fulfilled.
\end{example}

\begin{example}\label{ex4}\rm
Let $V(x)=\bigl(1+|x|^2\bigr)^{s/2}$, where $s\ge 1$.
Let $\Phi\colon \mathbb{R}^d\times[0, T]\times\mathbb{R}\to\mathbb{R}$ be a Borel function.
Assume that for every $t\in[0, T]$ the function $\Phi(x, t, r)$ is continuous in $(x, r)$.
For every $\mu\in\mathcal{M}(V)$ we define
$$
\varphi(x, t, \mu)=\Phi\Bigl(x, t, \int_0^T\int_{\mathbb{R}^d}\zeta(y)\mu_{\tau}(dy)\,d\tau\Bigr),
$$
where $\zeta\in C(\mathbb{R}^d)$ and $\lim_{|x|\to+\infty}\zeta(x)/V(x)=0$.
Let $R>0$. Then the $V$--weak convergence of measures $\mu^n=\mu^n_t\,dt\in\mathcal{M}_R(V)$
to a measure $\mu=\mu_t\,dt\in\mathcal{M}_R(V)$
implies that the equality
$$
\lim_{n\to\infty}\sup_{x\in B}\big|\varphi(x, t, \mu^n)-\varphi(x, t, \mu)\big|=0
$$
holds for every $t\in[0, T]$ and every open ball $B\subset\mathbb{R}^d$.

\begin{proof}
Let $C_{\zeta}=\sup_{x\in\mathbb{R}^d}|\zeta(x)|/V(x)$.
Suppose that measures $\mu^n=\mu^n_t\,dt\in\mathcal{M}_R(V)$ converge $V$--weakly to
a measure $\mu=\mu_t\,dt\in\mathcal{M}_R(V)$. Then for every $t\in[0, T]$ we have
$$
\lim_{n\to\infty}\int_{\mathbb{R}^d}\zeta(x)\mu_t^n(dx)=\int_{\mathbb{R}^d}\zeta(x)\mu_t(dx)
\quad \hbox{\rm and} \quad \Bigl|\int_{\mathbb{R}^d}\zeta(x)\mu^n_t(dx)\Bigr|\le C_{\zeta}R.
$$
By Lebesgue's dominated convergence theorem we obtain
$$
\lim_{n\to\infty}\int_0^T\int_{\mathbb{R}^d}\zeta(x)\mu_t^n(dx)\,dt=\int_0^T\int_{\mathbb{R}^d}\zeta(x)\mu_t(dx)\,dt.
$$
Let $\overline{B}$ denote the closure of a ball $B$. Since $\overline{B}\times[-C_{\zeta}R, C_{\zeta}R]$ is
a compact set, the function $\Phi(x, t, r)$ is uniformly continuous in $(x, r)$
on $\overline{B}\times[-C_{\zeta}R, C_{\zeta}R]$.
Therefore the equality
$$
\lim_{n\to\infty}\sup_{x\in B}\Bigg|\Phi\Bigl(x, t, \int_0^T\int_{\mathbb{R}^d}\zeta(y)\mu^n_{\tau}(dy)\,d\tau\Bigr)
-\Phi\Bigl(x, t, \int_0^T\int_{\mathbb{R}^d}\zeta(y)\mu_{\tau}(dy)\,d\tau\Bigr)\Bigg|=0
$$
holds for every $t\in[0, T]$.
\end{proof}

Thus if each of the functions $a^{ij}$, $b^i$, $q^{im}$, $g$ and $f$
is given by the same rule as the function~$\varphi$, then the conditions
(H1.2), (H1.3), (H3.3) are fulfilled.
\end{example}
	\vspace*{0.2cm}
	
	\section{\sc Auxiliary results}
	
	\vspace*{0.2cm}
	
	This section is devoted to the assertions playing the crucial role in the proofs of Theorem~\ref{th1} and Corollary~\ref{cor1}. At the beginning we briefly discuss the weak topology on the space of measures, the set $\mathcal{M}(V)$ and
    Fokker--Planck--Kolmogorov equations.

	\begin{remark}\label{r0}\rm
		Recall that the weak topology on the linear space of bounded Borel measures on $\mathbb{R}^d$
		is generated by the seminorms
		$$
		P_{\varphi}(\mu)=\int_{\mathbb{R}^d}\varphi(x)\mu(dx),
		$$
		where $\varphi$ is a bounded continuous function on $\mathbb{R}^d$.
		The weak topology on the space of probability measures is metrizable,
        for example, this topology is generated by the Kantorovich--Rubinshtein metric
		$$
		d(\mu, \sigma)=\sup\Bigl\{\int\varphi\,d(\mu-\sigma)\colon \, |\varphi(x)|\le 1, |\varphi(x)-\varphi(y)|\le|x-y|\Bigr\}.
		$$
		Similarly one can define the weak topology on the space of bounded Borel measures on the space $\mathbb{R}^d\times[0, T]$ and
this topology is metrizable on the set of nonnegative measures $\mu$ satisfying the equality~$\mu(\mathbb{R}^d\times[0, T])=T$. The sequence of probability measures $\mu^n$ on $\mathbb{R}^d$ is tight if for every
number $\varepsilon>0$ there exists a compact set $K_{\varepsilon}\subset\mathbb{R}^d$ such that
$\mu^n(K_{\varepsilon})\ge 1-\varepsilon$ for all $n$. Let $\mathcal{V}$ be a continuous nonnegative function with $\lim_{|x|\to\infty}\mathcal{V}(x)=+\infty$ and $\mu^n$ be a sequence of probability measures.
The estimate
$$
\sup_n\int_{\mathbb{R}^d}\mathcal{V}(x)\mu^n(dx)<\infty
$$
implies that the sequence $\mu^n$ is tight. According to Prokhorov's theorem, if the sequence $\mu^n$ is tight
then there exists a subsequence $\mu^{n_k}$ and a probability measure $\mu$ such that the measures~$\mu^{n_k}$
converge weakly to the measure $\mu$.
		
Let $\mu^n$ be a sequence of probability measures on $\mathbb{R}^d$. Assume that $\mu^n$ is tight.
If for every function $\psi\in C_{0}^{\infty}(\mathbb{R}^d)$
and every number $\varepsilon>0$ there exists a number $N$ such that the inequality
$$
\Bigl|\int_{\mathbb{R}^d}\psi(x)\mu^n(dx)-\int_{\mathbb{R}^d}\psi(x)\mu^k(dx)\Bigr|<\varepsilon,
$$
holds for all $n, k>N$, then the measures~$\mu^n$ weakly converge to some probability measure $\mu$.
Assertion follows from the Prokhorov theorem and the fact that for every $\psi\in C_0^{\infty}(\mathbb{R}^d)$
the sequence $\displaystyle\int_{\mathbb{R}^d}\psi(x)\mu^n(dx)$ converges.
Moreover, probability measures~$\mu^n$ on $\mathbb{R}^d$
converge weakly to a probability measure~$\mu$ whenever the equality
$$
\lim_{n\to\infty}\int_{\mathbb{R}^d}\psi(x)\mu^n(dx)=\int_{\mathbb{R}^d}\psi(x)\mu(dx)
$$
holds for every $\psi\in C_0^{\infty}(\mathbb{R}^d)$.

Let probability measures~$\mu^n$ on $\mathbb{R}^d$ converge weakly to a probability measure~$\mu$.
Assume that for a number $C>0$ and a nonnegative continuous function $\mathcal{W}$ one has the estimate
$$
\sup_n\int_{\mathbb{R}^d}\mathcal{W}(x)\,\mu^n(dx)\le C.
$$
Below we often use the fact that this estimate implies the inequality
$$
\int_{\mathbb{R}^d}\mathcal{W}(x)\,\mu(dx)\le C.
$$

The weak topology on the space of measures is discussed in
\cite{Bogach-wk} and \cite[Chapter 8]{tm2Bogach}.
	\end{remark}
	
	\begin{remark}\label{r1}\rm
		Let $R>0$ and
		$$
		\beta_{V, W}(R)=\sup_{\eta}R^{-1}\int_{\mathbb{R}^d}W(x)\,\eta(dx),
		$$
		where the supremum is taken over all probability measures $\eta$ satisfying the condition
		$$
		\int_{\mathbb{R}^d}V(x)\eta(dx)\le R.
		$$
		Let us prove that
		$$
		\lim_{R\to+\infty}\beta_{V, W}(R)=0.
		$$
		For every number $\varepsilon>0$ there exists a number $m_{\varepsilon}>0$ such that
		the equality $W(x)\le \varepsilon V(x)$ holds for all $x$ satisfying $V(x)>m_{\varepsilon}$.
		Then we obtain the inequality
		$$
		R^{-1}\int_{\mathbb{R}^d} W(x)\eta(dx)\le \varepsilon+m_{\varepsilon}R^{-1}.
		$$
	\end{remark}
	
	\begin{remark}\label{r2}\rm
		Let $R>0$.
		Suppose that we are given a sequence $\mu^n=\mu^n_t\,dt\in\mathcal{M}_R(V)$ and a measure $\mu=\mu_t\,dt\in\mathcal{M}_R(V)$.
		Assume that for every $t\in[0, T]$ the measures $\mu_t^n$ converge weakly to the measure $\mu_t$. Then the measures
		$\mu^n$ converge $V$--weakly to the measure $\mu$.
		
		Let $\zeta\in C(\mathbb{R}^d)$ and
		$\lim_{|x|\to\infty}|\zeta(x)|/V(x)=0$. For every natural number $N$ we set
		$$
		\zeta_N(x)=\max\Bigl\{-\frac{V(x)}{N}, \min\Bigl\{\frac{V(x)}{N}, \zeta(x)\Bigr\}\Bigr\}.
		$$
		Note that $\zeta_N$ is a continuous function, $|\zeta_N(x)|\le V(x)/N$ and $\zeta(x)=\zeta_N(x)$
		for sufficiently large~$|x|$. Since the function $\zeta-\zeta_N$ is continuous and bounded, we have
		$$
		\lim_{n\to\infty}\int_{\mathbb{R}^d}\Bigl(\zeta(x)-\zeta_N(x)\Bigr)\mu_t^n(dx)=
		\int_{\mathbb{R}^d}\Bigl(\zeta(x)-\zeta_N(x)\Bigr)\mu_t(dx).
		$$
		Taking into account the estimates
		$$
		\Bigl|\int_{\mathbb{R}^d}\zeta(x)\mu_t^n(dx)-
		\int_{\mathbb{R}^d}\Bigl(\zeta(x)-\zeta_N(x)\Bigr)\mu_t^n(dx)\Bigr|\le
		\frac{1}{N}\int_{\mathbb{R}^d}V(x)\mu_t^n(dx)\le \frac{R}{N}
		$$
		and
		$$
		\Bigl|\int_{\mathbb{R}^d}\zeta(x)\mu_t(dx)-
		\int_{\mathbb{R}^d}\Bigl(\zeta(x)-\zeta_N(x)\Bigr)\mu_t(dx)\Bigr|\le \frac{R}{N},
		$$
		we arrive at the equality
		$$
		\lim_{n\to\infty}\int_{\mathbb{R}^d}\zeta(x)\mu_t^n(dx)=
		\int_{\mathbb{R}^d}\zeta(x)\mu_t(dx).
		$$
	\end{remark}

	\begin{remark}\label{r3}\rm
		Suppose that we are given Borel functions $\alpha^{ij}(x, t)$ and $\beta^i(x, t)$ on $\mathbb{R}^d\times[0, T]$.
        Assume also that the matrix $\alpha=(\alpha^{ij})$ is symmetric and nonnegative definite.
        Let us consider the differential operator
		$$
		\mathcal{L}\psi(x, t)={\rm trace}\bigl(\alpha(x, t)D^2\psi(x)\bigr)+\langle\beta(x, t), \nabla\psi(x)\rangle.
		$$
		Let $\nu$ is a probability measure on $\mathbb{R}^d$.
		The key object in our considerations is a probability solution $\mu=\mu_t\,dt$ of the Cauchy problem for the
        Fokker--Planck--Kolmogorov equation
		$$
		\partial_t\mu_t=\mathcal{L}^{*}\mu_t, \quad \mu_0=\nu.
		$$
		We shall say that the measure $\mu=\mu_t\,dt$ is a probability solution on $[0, T]$
        if the measure $\mu$ is given by a family of probability measures $(\mu_t)_{t\in[0, T]}$ on $\mathbb{R}^d$ such that
		the mapping $t\mapsto \mu_t$ is continuous with respect to the weak topology, for every open ball
		$B\subset\mathbb{R}^d$ the functions $\alpha^{ij}$, $\beta^i$ are integrable on $B\times[0, T]$
        with respect to the measure $\mu=\mu_t\,dt$ and the equality
		$$
		\int_{\mathbb{R}^d}\psi(x)\mu_t(dx)-\int_{\mathbb{R}^d}\psi(x)\nu(dx)=
		\int_0^t\int_{\mathbb{R}^d}\mathcal{L}\psi(x, s)\mu_s(dx)\,ds \quad (*)
		$$
        holds for every $\psi\in C_0^{\infty}(\mathbb{R}^d)$ and all $t\in[0, T]$.

		The survey of the modern theory of Fokker--Planck--Kolmogorov equations is discussed in \cite{bookFPK}.
        Moreover, we essentially use the Ambrosio--Figalli--Trevisan superposition principle presenting
        in the papers \cite{Superp21}, \cite{Trev}.
		
        We need a modification of the estimates from \cite[Theorem 7.1.1]{bookFPK} and \cite[Lemma 2.2]{Superp21}.
        Assume that there exist a nonnegative function $\mathcal{V}\in C^2(\mathbb{R}^d)$ and
        nonnegative Borel function $\mathcal{W}$ on $\mathbb{R}^d\times[0, T]$ such that
		$$
		\lim_{|x|\to\infty}\mathcal{V}(x)=+\infty, \quad \mathcal{L}\mathcal{V}(x, t)\le \mathcal{W}(x, t)+C\mathcal{V}(x), \quad
		\int_0^T\int_{\mathbb{R}^d}\mathcal{W}(x, t)\mu_t(dx)\,dt<\infty,
		$$
		where the measure $\mu=\mu_t\,dt$ is a probability solution to the Cauchy problem
		$\partial_t\mu_t=\mathcal{L}^{*}\mu_t$, $\mu_0=\nu$ and $\mathcal{V}\in L^1(\nu)$.
		Then
		$$
		\int_{\mathbb{R}^d}\mathcal{V}(x)\mu_t(dx)\le \Bigl(\int_{\mathbb{R}^d}\mathcal{V}(x)\nu(dx)
		+\int_0^t\int_{\mathbb{R}^d}e^{-Cs}\mathcal{W}(x, s)\mu_s(dx)\,ds\Bigr)e^{Ct}.
		$$

		By \cite[Theorem 6.7.3]{bookFPK} it follows that if the coefficients $\alpha^{ij}$ and $\beta^i$
		are continuous in $x$ and bounded on $B\times[0, T]$ for every open ball $B\subset\mathbb{R}^d$
        then for every probability measure $\nu$ there exists a family of sub-probability measures
		$(\mu_t)_{t\in[0, T]}$ (that is $\mu_t\ge 0$ and $\mu_t(\mathbb{R}^d)\le 1$) such that for every Borel set $E$
        the mapping $t\mapsto\mu_t(E)$ is Borel measurable and for every function $\psi\in C_0^{\infty}(\mathbb{R}^d)$
		the equality ${\rm(*)}$ is fulfilled for almost all $t\in[0, T]$ .
        Moreover, by \cite[Theorem 7.1.1]{bookFPK} it follows that if there exists a function $\mathcal{V}\in C^2(\mathbb{R}^d)$ such that
		$$
		\lim_{|x|\to\infty}\mathcal{V}(x)=+\infty, \quad \mathcal{L}\mathcal{V}(x, t)\le C\mathcal{V}(x)+C, \quad
		\mathcal{V}\in L^1(\nu),
		$$
		then $\mu_t(\mathbb{R}^d)=1$ and $\|\mathcal{V}\|_{L^1(\mu_t)}\le C'$ for almost all $t\in[0, T]$,
        where the constant $C'$ does not depend on $t$. According to the \cite[Proposition 4.1]{izvest25},
        one can redefine the solution $\mu_t$ in such a way that the equality (*) is fulfilled for all $t\in[0, T]$.
        Note that for every $\psi\in C_0^{\infty}(\mathbb{R}^d)$ there exists a number $C(\psi)>0$ such that
        the estimate
        $$
        \Bigl|\mathcal{L}\psi(x, t)\Bigr|\le C(\psi)
        $$
        holds for all $(x, t)\in\mathbb{R}^d\times[0, T]$. Using the integral identity (*),
        we derive the inequality
        $$
		\Bigl|\int_{\mathbb{R}^d}\psi(x)\,d\mu_t-\int_{\mathbb{R}^d}\psi(x)\,d\mu_s\Bigr|\le C(\psi)|t-s|, \quad s, t\in[0, T].
		$$
		This inequality implies that $\mu_t(\mathbb{R}^d)=1$ for all $t\in[0, T]$ and the mapping $t\mapsto\mu_t$ is continuous.
	\end{remark}
	
\begin{remark}\label{rem-int}\rm
In Theorem~\ref{th1} it is assumed the integrability of the function $h(|v(x, t)|)$
with respect to the measure $\sigma=\sigma_t\,dt$, where $\sigma=\sigma_t\,dt$ is a probability solution
to the Cauchy problem
$$
\partial_t\sigma_t=L_{\mu, v(x, t)}^{*}\sigma_t, \quad \sigma_0=\nu.
$$
However it is not assumed the condition
$$
\sup_{t\in[0, T]}\int_{\mathbb{R}^d}V(x)\sigma_t(dx)<\infty.
$$
Let us remark that, according to the condition (H2.1), there exists a number $C(\mu)>0$ such that
$$
L_{\mu, v(x, t)}V(x, t)\le C_LV(x)+C(\mu)+h(|v(x, t)|).
$$
By the previous remark we obtain the estimate
$$
\int_{\mathbb{R}^d}V(x)\sigma_t(dx)\le \Bigl(\int_{\mathbb{R}^d}V(x)\nu(dx)
+\int_0^T\int_{\mathbb{R}^d}h(|v(x, s)|)\sigma_s(dx)\,ds+C(\mu)T\Bigr)e^{C_Lt}.
$$
Thus the integrability of the function $(x, t)\mapsto h(|v(x, t)|)$ implies $\sigma\in\mathcal{M}(V)$.
\end{remark}

	We need the following a priori estimates.
	
	\begin{lemma}\label{lem1}
		Suppose that the conditions $\rm (H1), (H2), (H3)$ are fulfilled and $\nu$ is a probability measure on
        $\mathbb{R}^d$ such that $V\in L^1(\nu)$. Let $\alpha$ be a nonnegative function on $[0, +\infty)$ and $\lim_{R\to+\infty}\alpha(R)=0$.
		Then for all $\gamma>0$ and $M\ge 5C_L$ there exists a number $R_0>0$ such that for every $R>R_0$ and $t\in[0, T]$
        the estimates
		$$
		\int_{\mathbb{R}^d} V\,d\mu_t\le Re^{Mt}, \quad
		\int_0^{T}\int_{\mathbb{R}^d}h(|u(x, t)|)\,d\mu_t\,dt\le\gamma R.
		$$
		are fulfilled for every Borel mapping $u\colon \mathbb{R}^d\times[0, T]\to U$ and
		for every measure $\mu\in \mathcal{M}(V)$ satisfying the following conditions:
		{\rm 1)} the measure $\mu=\mu_t\,dt$ is a probability solution to the Cauchy problem
        $\partial_t\mu=L^{*}_{\sigma, u(x, t)}\mu$, $\mu_0=\nu$, where $\sigma\in\mathcal{M}_{R,M}(V)$,
		{\rm 2)} the estimate
		$$
		\int_0^{T}\int_{\mathbb{R}^d}f(u(x, t), x, t, \sigma)\,\mu_t(dx)\,dt
		+\int_{\mathbb{R}^d}g(x, \tau, \sigma)\mu_{\tau}(dx)\le \alpha(Re^{MT})Re^{MT}
		$$
		holds. Note that the constant $R_0$ depends only on the function $\alpha$ and the numbers
        $\gamma$, $M$, $C_f$, $C_g$, $C_L$, $T$, $\|V\|_{L^1(\nu)}$.
	\end{lemma}
	\begin{proof}
For every $\sigma\in\mathcal{M}_{R,M}(V)$ the inequality
$$
\int_{\mathbb{R}^d}V(x)\sigma_t(dx)\le Re^{MT}
$$
holds for every $t\in[0, T]$. Taking into account the definition of the function $\beta_{V, W}$
from Remark~\ref{r1} we obtain
$$
\sup_{t\in[0, T]}\int_{\mathbb{R}^d}W(x)\sigma_t(dx)\le \beta_{V, M}(Re^{MT})Re^{MT}.
$$
		Applying the inequalities from the conditions (H3.1) and (H3.2), we arrive at the estimate
		\begin{multline*}
			\int_0^{T}\int_{\mathbb{R}^d}h(|u(x, t)|)\,d\mu_t\,dt\le
			\alpha(Re^{MT})Re^{MT}+
\\
C_f\int_0^{T}\int_{\mathbb{R}^d}W(x)\,\mu_t(dx)\,dt+C_fT\beta_{V, W}(Re^{MT})Re^{MT}+
			\\
			C_g\int_{\mathbb{R}^d}W(x)\,\mu_{T}(dx)+C_g\beta_{V, W}(Re^{MT})Re^{MT}.
		\end{multline*}

Let $\varepsilon>0$.		
Arguing as in Remark~\ref{r1}, we find a number
$m_{\varepsilon}>0$ such that the inequality $W(x)\le \varepsilon V(x)$ holds
for all $x$ satisfying the inequality $V(x)>m_{\varepsilon}$. Furthermore, there exists $R_1>0$
such that for all $R>R_1$ one has the estimates
$$
\alpha(Re^{MT})e^{MT}<\varepsilon \quad \hbox{\rm and} \quad \beta_{V, W}(Re^{MT})e^{MT}<\varepsilon.
$$
		Then we obtain
		\begin{multline*}
			\int_0^{T}\int_{\mathbb{R}^d}h(|u(x, t)|)\,d\mu_t\,dt\le
			\varepsilon(1+C_fT+C_g\bigr)R+m_{\varepsilon}(C_fT+C_g)+
			\\
			\varepsilon C_f\int_0^T\int_{\mathbb{R}^d}V(x)\mu_t(dx)\,dt+\varepsilon C_g\int_{\mathbb{R}^d}V(x)\mu_T(dx).
		\end{multline*}
        Let us remark that
        \begin{multline*}
        \langle Q(x, t, \sigma)u(x, t), \nabla V(x)\rangle\le \big|Q^{\top}(x, t, \sigma)\nabla V(x)\bigr|\big|u(x, t)\big|\le
        \\
        \le h^{*}\Bigl(\big|Q^{\top}(x, t, \sigma)\nabla V(x)\big|\Bigr)+h\Bigl(\big|u(x, t)\big|\Bigr).
        \end{multline*}
		Using the condition (H2.1), we get the inequalities
		\begin{multline*}
		L_{\sigma, u(x, t)}V(x, t)\le h(|u(x, t)|)+C_L V(x) + C_L\int_{\mathbb{R}^d}V(y)\sigma_t(dy)
+C_L\sup_{t\in[0, T]}\int_{\mathbb{R}^d}W(y)\sigma_t(dy)\le
\\
h(|u(x, t)|)+C_L V(x) + C_L\int_{\mathbb{R}^d}V(y)\sigma_t(dy)+C_L\beta_{V, W}(Re^{MT})Re^{MT}
		\end{multline*}
		which implies the estimate
		$$
		L_{\sigma, u(x, t)}V(x, t)\le h(|u(x, t)|)+C_L V(x)+C_LRe^{Mt}+\varepsilon C_L R.
		$$
        Taking into account the estimate from Remark~\ref{r3}, we obtain
		$$
		\int_{\mathbb{R}^d}V\,d\mu_t\le e^{C_Lt}\Bigl(\int_{\mathbb{R}^d}V(x)\nu(dx)+
\int_0^T\int_{\mathbb{R}^d}h(|u(x, s)|)\mu_s(dx)\,ds+\varepsilon C_LRT\Bigr)+\frac{C_L}{M-C_L}Re^{Mt}.
		$$
Since $M\ge 5C_L$, we have
$$
\frac{C_L}{M-C_L}\le\frac{1}{4}.
$$
Set
$$
\theta=\Bigl(1+C_LT+C_fT+C_g+C_f+\int_{\mathbb{R}^d}V(x)\nu(dx)\Bigr)e^{C_LT}.
$$
Then we obtain
\begin{multline*}
\int_{\mathbb{R}^d}V(x)\mu_t(dx)\le \frac{1}{4}Re^{Mt}+\bigl(1+m_{\varepsilon}\bigr)\theta
+\varepsilon\theta R
\\
+\varepsilon\theta\int_0^T\int_{\mathbb{R}^d} V(x)\mu_t(dx)\,dt
+\varepsilon\theta\int_{\mathbb{R}^d} V(x)\mu_T(dx).	
\end{multline*}
Integrating this inequality with respect to $t$, we get the estimate
\begin{multline*}
\int_0^T\int_{\mathbb{R}^d}V(x)\mu_t(dx)\,dt\le \frac{1}{4M}Re^{MT}+\bigl(1+m_{\varepsilon}\bigr)\theta T
+\varepsilon\theta RT
\\
+\varepsilon\theta T\int_0^T\int_{\mathbb{R}^d} V(x)\mu_t(dx)\,dt
+\varepsilon\theta T\int_{\mathbb{R}^d} V(x)\mu_T(dx).	
\end{multline*}
Summing this inequality and the previous inequality with $t=T$, we obtain
\begin{multline*}
\int_0^T\int_{\mathbb{R}^d}V(x)\mu_t(dx)\,dt+
\int_{\mathbb{R}^d}V(x)\mu_T(dx)\le \frac{M+1}{4M}Re^{MT}+\bigl(1+m_{\varepsilon}\bigr)(1+T)\theta
+\varepsilon\theta R(1+T)
\\
+\varepsilon\theta(1+T)\int_0^T\int_{\mathbb{R}^d} V(x)\mu_t(dx)\,dt
+\varepsilon\theta(1+T)\int_{\mathbb{R}^d} V(x)\mu_T(dx).	
\end{multline*}
Choosing $\varepsilon>0$ such that $\varepsilon\theta(1+T)\le 1/2$, we derive the estimate
$$
\int_0^T\int_{\mathbb{R}^d}V(x)\mu_t(dx)\,dt+
\int_{\mathbb{R}^d}V(x)\mu_T(dx)\le \frac{M+1}{2M}Re^{MT}+2\bigl(1+m_{\varepsilon}\bigr)(1+T)\theta
+2\varepsilon\theta R(1+T).
$$
Hence for all $t\in[0, T]$ we have
\begin{multline*}
\int_{\mathbb{R}^d}V(x)\mu_t(dx)\le \frac{1}{4}Re^{Mt}+\bigl(1+m_{\varepsilon}\bigr)\theta
+\varepsilon\theta R
\\
+\varepsilon\theta\Bigl(\frac{M+1}{2M}Re^{MT}+2\bigl(1+m_{\varepsilon}\bigr)(1+T)\theta
+2\varepsilon\theta R(1+T)\Bigr).	
\end{multline*}
Choosing $\varepsilon>0$ so small that
$$
\varepsilon\theta\Bigl(1+\frac{M+1}{2M}e^{MT}+2\varepsilon\theta(1+T)\Bigr)\le\frac{1}{4},
$$
we obtain the inequality
$$
\int_{\mathbb{R}^d}V(x)\mu_t(dx)\le \frac{1}{4}Re^{Mt}+\frac{1}{4}R+\bigl(1+m_{\varepsilon}\bigr)\theta
+2\varepsilon\theta^2(1+m_{\varepsilon})(1+T).
$$
There exists a number $R_2>R_1$ such that for every $R>R_2$ and all $t\in[0, T]$ the estimate
$$
\int_{\mathbb{R}^d}V(x)\mu_t(dx)\le Re^{Mt}
$$
holds. Applying this estimate, we obtain
$$
			\int_0^{T}\int_{\mathbb{R}^d}h(|u(x, t)|)\,d\mu_t\,dt\le
			\varepsilon(1+C_fT+C_g\bigr)R+m_{\varepsilon}(C_fT+C_g)+
			\varepsilon C_f M^{-1}Re^{MT}+\varepsilon C_g R e^{MT}.
$$
Choosing $\varepsilon>0$ small enough we get
$$
\varepsilon(1+C_fT+C_g\bigr)+\varepsilon C_fM^{-1}e^{MT}+\varepsilon C_g e^{MT}\le\frac{\gamma}{2}.
$$
Then the estimate
$$
			\int_0^{T}\int_{\mathbb{R}^d}h(|u(x, t)|)\,d\mu_t\,dt\le \frac{\gamma}{2}R+
			m_{\varepsilon}(C_fT+C_g)
$$
holds. There exists a number $R_0>R_2$ such that for all $R>R_0$ we have
$$
\int_0^{T}\int_{\mathbb{R}^d}h(|u(x, t)|)\,d\mu_t\,dt\le \gamma R
$$
Note that the number $\varepsilon$ depends only on the numbers
$M$, $\gamma$, $C_f$, $C_g$, $C_L$, $T$ and $\|V\|_{L^1(\nu)}$.
The numbers $R_0$, $R_1$ and $R_2$ depends on $\varepsilon$, the functions $\alpha$ and $\beta_{V, W}$.		
\end{proof}

	Below we shall use some special compact sets
    in the space of bounded nonnegative Borel measures $\mu$ on $\mathbb{R}^d\times[0, T]$ satisfying the equality
	$\mu(\mathbb{R}^d\times[0, T])=T$.
	
	Suppose that for every function $\psi\in C_0^{\infty}(\mathbb{R}^d)$ we are given a continuous
	nondecreasing function~$\omega_{\psi}$ on $[0, T]$ such that $\omega_{\psi}(0)=0$.
	Set $\omega=\{\omega_{\psi}\}$.
    Let $R>0$ and $M\ge 0$. Denote by $\mathcal{M}_{R, M}^{\omega}(V)$ the set of measures
    $\mu=\mu_t\,dt\in\mathcal{M}_{R, M}(V)$ such that
	for every function $\psi\in C_0^{\infty}(\mathbb{R}^d)$ the inequality
	$$
	\Bigl|\int_{\mathbb{R}^d}\psi\,d\mu_t-\int_{\mathbb{R}^d}\psi\,d\mu_s\Bigr|\le \omega_{\psi}(|t-s|)
	$$
    holds for all $t, s\in[0, T]$. Note that $\mathcal{M}_{R, M}^{\omega}(V)$ is a convex set in the
    space of bounded nonnegative Borel measures $\mu$ on $\mathbb{R}^d\times[0, T]$
    satisfying the condition $\mu(\mathbb{R}^d\times[0, T])=T$.
	
	\begin{lemma}\label{lem2}
		{\rm (i)} The class $\mathcal{M}_{R, M}^{\omega}(V)$ is a compact set in the weak topology
        on the space of bounded nonnegative Borel measures $\mu$ on $\mathbb{R}^d\times[0, T]$
        satisfying the condition $\mu(\mathbb{R}^d\times[0, T])=T$.
		
		{\rm (ii)}
		If measures $\mu^n\in \mathcal{M}_{R, M}^{\omega}(V)$
		converge weakly to a measure $\mu\in\mathcal{M}_{R, M}^{\omega}(V)$, then
		the measures $\mu^n$ converge $V$--weakly to the measure $\mu$.
	\end{lemma}
	\begin{proof}
		Let us prove (i). Let $\mu^n=\mu_t^n\,dt\in \mathcal{M}_{R, M}^{\omega}(V)$.
		Denote by $\mathbb{Q}$ the set of rational numbers. Note that for every $n$ the estimate
		$$
		\int_{\mathbb{R}^d}V(x)\,d\mu_t^n\le Re^{Mt}
		$$
		holds for all $t\in[0, T]$. Applying the Prokhorov theorem and the diagonal procedure, we find a sequence
        of numbers $n_j$ such that for all $t\in \mathbb{Q}\cap [0, T]$
        the measures $\mu_t^{n_j}$ converge weakly to a probability measure $\mu_t$.
		For every function $\psi\in C^{\infty}_0(\mathbb{R}^d)$ and all $s, t\in[0, T]$ the estimate
		$$
		\Bigl|\int_{\mathbb{R}^d}\psi\,d\mu_t^{n_j}-\int_{\mathbb{R}^d}\psi\,d\mu_s^{n_j}\Bigr|\le
		\omega_{\psi}(|t-s|)
		$$
        holds. Take a number $t\in[0, T]$. For  every $\varepsilon>0$ there exists a rational number
		$r$ such that
		$$
		\Bigl|\int_{\mathbb{R}^d}\psi\,d\mu_t^{n_j}-\int_{\mathbb{R}^d}\psi\,d\mu_r^{n_j}\Bigr|\le \varepsilon.
		$$
		There exists a natural number $N$ such that for all $j, k>N$ one hase the estimate
		$$
		\Bigl|\int_{\mathbb{R}^d}\psi\,d\mu_r^{n_j}-\int_{\mathbb{R}^d}\psi\,d\mu_r^{n_k}\Bigr|\le \varepsilon.
		$$
		Thus for all $j, k>N$ we obtain the inequality
		$$
		\Bigl|\int_{\mathbb{R}^d}\psi\,d\mu_t^{n_j}-\int_{\mathbb{R}^d}\psi\,d\mu_t^{n_k}\Bigr|\le 4\varepsilon.
		$$
		According to Remark~\ref{r0}, the measures $\mu^{n_j}_t$ converge weakly to a probability measure $\mu_t$.
		Since for every $j$ the integral of $V$ with respect to the measure $\mu^{n_j}_t$ is majorized by $Re^{Mt}$,
        the same estimate holds for the integral of $V$ with respect to the measure $\mu_t$.
		Thus we obtain the family of probability measures $(\mu_t)_{t\in[0, T]}$. Note that
		for every $\psi\in C_0^{\infty}(\mathbb{R}^d)$ the estimate
		$$
		\Bigl|\int_{\mathbb{R}^d}\psi\,d\mu_t-\int_{\mathbb{R}^d}\psi\,d\mu_s\Bigr|\le
		\omega_{\psi}(|t-s|)
		$$
		holds for all $s, t\in[0, T]$. This estimate implies the continuity of the mapping
		$$
		t\mapsto\int_{\mathbb{R}^d}\psi\,d\mu_t.
		$$
		Therefore the mapping $t\mapsto\mu_t$ is continuous with respect to the weak topology. Set $\mu=\mu_t\,dt$.
		According to Remark~\ref{r2}, the measures $\mu^{n_j}$ converge $V$--weakly to the measure $\mu$.
		Finally, we note that for every bounded continuous function $\zeta$ on $\mathbb{R}^d\times[0, T]$
		one has the equality
		$$
		\lim_{j\to\infty}\int_{0}^{T}\int_{\mathbb{R}^d}\zeta(x, t)\,\mu_t^{n_j}(dx)\,dt=
		\int_{0}^{T}\int_{\mathbb{R}^d}\zeta(x, t)\,\mu_t(dx)\,dt.
		$$
		Thus the measures $\mu^{n_j}$ converge weakly to the measure $\mu$.
		
        Let us prove (ii).
        Suppose that measures $\mu^n=\mu^n_t\,dt\in\mathcal{M}_{R, M}(V)$
        converge weakly to a measure $\mu=\mu_t\,dt\in\mathcal{M}_{R, M}(V)$.
		It suffices to show that for every $t\in[0, T]$ the sequence $\mu^n_t$
        converges weakly to the measure~$\mu_t$. Note that every sequence of natural numbers $n_k$ posses a subsequence
		$n_{k_j}$ such that for every $t$ the measures $\mu^{n_{k_j}}_t$ converge weakly to a probability measure $\widetilde{\mu}_t$.
        In addition, the mapping $t\mapsto\widetilde{\mu}_t$ is continuous.
		Since $\mu_t\,dt=\widetilde{\mu}_t\,dt$, for every bounded and continuous function $\varphi$ the equality
        $$
        \int_{\mathbb{R}^d}\varphi\,d\mu_t=\int_{\mathbb{R}^d}\varphi\,d\widetilde{\mu}_t
        $$
        holds for almost all $t\in[0, T]$. The continuity of the mappings $t\mapsto\mu_t$ and $t\mapsto\widetilde{\mu}_t$ implies
		that the last equality is fulfilled for every $t\in[0, T]$. Therefor $\mu_t=\widetilde{\mu}_t$ for every $t\in[0, T]$.
		Thus for every $t$ the measures $\mu_t^n$ converge weakly to the measure $\mu_t$.
	\end{proof}

Let us state for future reference two additional remarks about the weak convergence and conditional measures.
		
	\begin{remark}\label{r4}\rm
		Assume that we are given complete separable metric spaces $X$ and $Y$.
        Let $P_n$ be a sequence of Borel probability measures on $X\times Y$ converging weakly to a
        probability measure $P$ on $X\times Y$.
		Suppose that $P_n(X\times B)=P(X\times B)=\pi(B)$ for every $n$. Then for every bounded Borel function
        $\eta$ on $X\times Y$ that is continuous in $x\in X$ we have the equality
		$$
		\lim_{n\to\infty}\int_{X\times Y}\eta(x, y)P_n(dxdy)=\int_{X\times Y}\eta(x, y)P(dxdy).
		$$
		This statement is known but for the reader's convenience
        we give a brief proof. We may assume that $|\eta|\le 1$.
		Let $\varepsilon>0$. Since the projections of measures $P_n$ on $X$ converge weakly to the projection
        of the measure $P$ on $X$,
		there exists a compact set $K\subset X$ such that the estimate $P_{n}(K\times Y)\ge 1-\varepsilon$
		holds for all~$n$.
		According to the Scorza Dragoni theorem (see, for instance, \cite[Theorem 2]{SkorcDr} or \cite[Theorem 7.14.26]{tm2Bogach}),
        there exists a compact set $C\subset Y$ such that $\pi(C)\ge 1-\varepsilon$ and
		the restriction of the function $\eta$ on $X\times C$ is continuous.
		Let $\widetilde{\eta}$ be a continuous function on $X\times Y$
        such that $\widetilde{\eta}=\eta$ on $K\times C$ and $|\widetilde{\eta}|\le 1$.
		We have
		$$
		\Bigl|\int_{X\times Y}\bigl(\eta-\widetilde{\eta}\bigr)\,dP_n\Bigr|\le
		2P_n(X\times Y\setminus K\times C)\le 2P_n((X\setminus K)\times Y)+2P_n(X\times(Y\setminus C))\le
		4\varepsilon.
		$$
		The analogous inequality holds for $P$. Since the measures $P_n$ converge weakly to $P$, there exists a number $N$
        such that for all $n>N$ the inequality
        $$
        \Big|\int_{X\times Y}\widetilde{\eta}\,dP_n-\int_{X\times Y}\widetilde{\eta}\,dP\Big|\le\varepsilon
        $$
        holds. Thus for all $n>N$ we obtain
        $$
        \Big|\int_{X\times Y}\eta\,dP_n-\int_{X\times Y}\eta\,dP\Big|\le 9\varepsilon.
        $$
        This concludes the proof.

	    Let us consider an example. Let $R>0$ and $M\ge 0$. Suppose that a Borel function $v(x, t)$ is continuous in $x$,
		for some number $C>0$ the estimate $|v(x, t)|\le C+CW(x)$ holds for all $(x, t)\in\mathbb{R}^d\times[0, T]$
        and measures $\mu_t^n\,dt\in\mathcal{M}_{R, M}(V)$ converge weakly to a measure $\mu_t\,dt\in\mathcal{M}_{R, M}(V)$.
		Arguing as in Remark~\ref{r2} and using the last assertion, we obtain
		$$
		\lim_{n\to\infty}\int_0^T\int_{\mathbb{R}^d}v(x, t)\mu_t^n(dx)\,dt=\int_0^T\int_{\mathbb{R}^d}v(x, t)\mu_t(dx)\,dt.
		$$
	\end{remark}

	\begin{remark}\label{r5}\rm
		Suppose that we are given complete separable metric spaces $X$ and $Y$.
        Let $\mu$ be a bounded Borel nonnegative measure on $X\times Y$. Denote by
		$\mu_Y$ the projection of $\mu$ on $Y$ that is $\mu_Y(E)=\mu(X\times E)$ for every Borel set $E$.
		By \cite[Theorem 10.4.10]{tm2Bogach} there exists a family of Borel probability measures $\mu^{y}$
        on $X$ such that for every Borel set $E$
		the mapping $y\mapsto \mu^y(E)$ is Borel measurable and for every Borel function $f\in L^1(\mu)$ the equality
		$$
		\int_{X\times Y}f(x, y)\mu(dxdy)=\int_Y\Bigl(\int_Xf(x, y)\mu^y(dx)\Bigr)\mu_Y(dy).
		$$
		holds. Recall that the measures $\mu^y$ are called conditional measures.
		
		Assume that we are given a family of Borel probability measures $(\sigma^y)_{y\in Y}$ on $X$.
		The mapping $y\mapsto\sigma^y(E)$ is Borel measurable for every Borel set $E$ if and only if the mapping
		$$
		y\mapsto \int_X\varphi(x)\sigma^{y}(dx)
		$$
		is Borel measurable for every bounded continuous function $\varphi$ on $X$.
        In this case we say that the family $(\sigma^y)_{y\in Y}$ is Borel measurable in $y$.

        Let us consider a Borel measurable family $(\sigma^y)_{y\in Y}$ on $X$ and a nonnegative bounded Borel measure
        $\eta$ on $Y$. According to \cite[Theorem 10.7.2]{tm2Bogach}, we can define the bounded nonnegative Borel measure $\mu$ by the equality
        $\mu(dxdy)=\sigma^{y}(dx)\eta(dy)$ which means that for every Borel set~$C$
		$$
		\mu(C)=\int_{X}\Bigl(\int_Y I_C(x, y)\sigma^y(dx)\Bigr)\eta(dy),
		$$
		where $I_C$ is the indicator of the set $C$.
		Note that the measures $\sigma^y$ are conditional measures for the measure $\mu$.

        Let us consider an example.
		Assume that the measure $\mu$ on $\mathbb{R}^d\times[0, T]$ is given by a family
		of Borel probability measures $\mu_t$ on $\mathbb{R}^d$ and the mapping
		$t\mapsto \mu_t$ is continuous with respect to weak topology.
        Suppose that we are given a Borel function $(x, t)\mapsto u(x, t)$ from~$\mathbb{R}^d\times[0, T]$~to~$U$.
		Then we can define the measure
		$$
		\Pi(dudxdt)=\delta_{u(x, t)}(du)\mu_t(dx)\,dt
		$$
        on the space $U\times\mathbb{R}^d\times[0, T]$.
		For every $t\in[0, T]$ the measure $\delta_{u(x, t)}(du)\mu_t(dx)$ is well defined since the function
		$x\mapsto \delta_{u(x, t)}(E)=I_E(u(x, t))$
        is Borel measurable for every Borel set $E\subset U$.
		Let us prove that the family of measures $\bigl(\delta_{u(x, t)}(du)\mu_t(dx)\bigr)_{t\in[0, T]}$ is Borel measurable in~$t$.
        Note that for every Borel set $C\subset U\times\mathbb{R}^d$ we have the equality
        $$
        \int_{U\times\mathbb{R}^d}I_C(u, x)\delta_{u(x, t)}(du)\mu_t(dx)=\int_{\mathbb{R}^d}I_C(u(x, t), x)\mu_t(dx),
        $$
        where $I_C$ is the indicator of $C$. According to \cite[Theorem 10.7.2]{tm2Bogach},
        the mapping
        $$
        t\mapsto \int_{\mathbb{R}^d}I_C(u(x, t), x)\mu_t(dx)
        $$
        is Borel measurable.
        Thus the measure $\Pi$ is well defined.
		Conditional measures are discussed in \cite[Chapter 10]{tm2Bogach}.
	\end{remark}
	
	Let
    $$
    R>0, \quad M=5C_L, \quad \gamma=\frac{1}{4}e^{-C_LT},
    $$
    where $C_L$ is a constant from the condition (H2.1).
    Let $P_{R}$ denote the set of all bounded nonnegative Borel measures
	$\Pi$ on $U\times\mathbb{R}^d\times[0, T]$ satisfying the conditions:

    1) $\Pi(U\times\mathbb{R}^d\times[0, T])=T$,

    2) the inequality
	$$
	\int_{U\times\mathbb{R}^d\times[0, t]} V(x)\,\Pi(dudxds)\le Re^{Mt}
    $$
    holds for all $t\in[0, T]$,

    3) the estimate
    $$
	\int_{U\times\mathbb{R}^d\times[0, T]}h(|u|)\,\Pi(dudxdt)\le \gamma R
	$$
    is fulfilled.

	If $V(0)T\le R$ and $h(|u_0|)T\le \gamma R$ for some $u_0\in U$,
	then the measure
	$$\Pi(dudxdt)=\delta_{u_0}(du)\otimes\delta_0(dx)\,dt$$
	belongs to the set $P_{R}$.
	Note that the set $P_{R}$ is compact in the weak topology (see Remark~\ref{r0}).
	
	Let $\sigma\in\mathcal{M}_{R, M}(V)$. Assume that $\nu$ is a Borel probability
    measure on $\mathbb{R}^d$ such that $V\in L^1(\nu)$.
	Denote by $S_{R}(\sigma)$ the set of measures $\Pi\in P_R$ such that
	the projection of $\Pi$ on $(x, t)$ is a measure $\mu=\mu_t\,dt\in \mathcal{M}_{R, M}(V)$ satisfying the conditions:
    1) $\mu_0=\nu$ and 2) the equality	
    $$
	\int_{\mathbb{R}^d}\psi\,d\mu_t=\int_{\mathbb{R}^d}\psi\,d\nu+\int_{U\times\mathbb{R}^d\times[0, t]}L_{\sigma, u}\psi\,d\Pi
	$$
	holds for all $t\in[0, T]$ and every function $\psi\in C_0^{\infty}(\mathbb{R}^d)$.
	
	\begin{remark}\label{r7}\rm
		Suppose that $\Pi\in S_{R}(\sigma)$ and $\mu=\mu_t\,dt$ is a projection of $\Pi$ on $(x, t)$.
		Denote by $\Pi_{x, t}(du)$ the conditional measures for $\Pi$ with respect to the projection~$\mu$.
		Since $\Pi\in S_{R}(\sigma)$, the functions $h(|u|)$ and $|u|$ (by the condition $\lim_{v\to\infty}h(v)/v=\infty$)
        are integrable with respect to the measure $\Pi$. Recall that $U$ is a convex closed set. Hence
        there exists a Borel function $(x, t)\mapsto u(x, t)$ from $\mathbb{R}^d\times[0, T]$ to $U$ such that
		for $\mu$ -- almost all $(x, t)$ the equality
		$$
		u(x, t)=\int_{U}u\,\Pi_{x, t}(du)
		$$
        holds.
		If $T=1$, then $\Pi$ is a probability measure and the function $(x, t)\mapsto u(x, t)$ is
        the conditional expectation of $\xi(u)=u$ with respect to the measure $\Pi(dudxdt)$
        and the sigma---algebra generating by the variables $x$ and~$t$.
		
		Applying Jensen's inequality, we obtain
		$$
		\int_0^T\int_{\mathbb{R}^d}h(|u(x, t)|)\mu_t(dx)\,dt\le
		\int_0^T\int_{\mathbb{R}^d}\int_Uh(|u|)\Pi_{x, t}(du)\mu_t(dx)\,dt\le\gamma R.
		$$
		Moreover, for every function $\psi\in C_0^{\infty}(\mathbb{R}^d)$ the equality
		$$
		\int_{U\times\mathbb{R}^d\times[0, t]}\langle Q(x, t, \sigma)u, \nabla\psi(x)\rangle\Pi(dudxdt)=
		\int_0^t\int_{\mathbb{R}^d}\langle Q(x, t, \sigma)u(x, t), \nabla\psi(x)\rangle\mu_t(dx)\,dt
		$$
        holds for all $t\in[0, T]$. Therefor the measure $\mu=\mu_t\,dt$ is a solution to the Cauchy problem
        $\partial_t\mu_t=L_{\sigma, u(x, t)}^{*}\mu_t$, $\mu_0=\nu$.
	\end{remark}
	
	The following two lemmas play a crucial role in the proof of the main results.
	
	\begin{lemma}\label{lem3}
		Suppose that the conditions {\rm (H1), (H2), (H3)} are fulfilled.
		Then there exists a number $R_0>0$ such that
		for every $R>R_0$ we have
		
		{\rm (i)} for every
		$\sigma\in\mathcal{M}_{R, M}(V)$ the set $S_{R}(\sigma)$ is nonempty,
		
		{\rm (ii)} there exists $\omega=\{\omega_{\psi}\}$ such that for every $\Pi\in S_R(\sigma)$
		the projection of $\Pi$ on $(x, t)$
		belongs to the set $\mathcal{M}_{R, M}^{\omega}(V)$,
		
		{\rm (iii)}
		the graph of the mapping $\sigma\mapsto S_{R}(\sigma)$ is closed in the space
		$\mathcal{M}_{R, M}^{\omega}(V)\times P_{R}$, in particular, the set $S_{R}(\sigma)$ is compact in the weak topology.
	\end{lemma}
	\begin{proof}
		Let us prove (i). Let $\sigma\in\mathcal{M}_{R, M}(V)$ and $u_0\in U$. According to the conditions (H1.1) and (H1.2),
		the coefficients $a^{ij}(x, t, \sigma)$ and $b^i(x, t, \sigma)$ are locally bounded and continuous in~$x$.
		By the condition (H2.1) we have the estimate
		$$
		L_{\sigma, u_0}V(x)\le C_LV(x)+h(|u_0|)+C_LRe^{Mt}+C_L\beta_{V, W}(Re^{MT})Re^{MT},
		$$
        where $\beta_{V, W}$ is defined in Remark~\ref{r1}.
		According to Remark~\ref{r3}, there exists a continuous mapping $t\mapsto\mu_t$
		such that the measure $\mu=\mu_t\,dt$ is a solution to the Cauchy problem
		$$
		\partial_t\mu_t=L_{\sigma, u_0}^{*}\mu_t, \quad \mu_0=\nu.
		$$
        Applying the estimates from Remark~\ref{r3} with $C=C_L$, $\mathcal{V}=V$ and
        $$
        \mathcal{W}(x, t)=h(|u_0|)+C_LRe^{Mt}+C_L\beta_{V, W}(Re^{MT})Re^{MT},
        $$
		we obtain the inequality
		$$
		\int_{\mathbb{R}^d}V(x)\,\mu_t(dx)\le e^{C_LT}\Bigl(\|V\|_{L^1(\nu)}+h(|u_0|)T
+C_L T\beta_{V, W}(Re^{MT})Re^{MT}\Bigr)+\frac{1}{4}Re^{Mt}.
		$$
        Here we also use the equality $M=5C_L$.
		There exists a number $R_0>0$ such that for all $R>R_0$ the estimates
		$$
		h(|u_0|)T\le \gamma R,
\quad e^{C_LT}\Bigl(\|V\|_{L^1(\nu)}+h(|u_0|)T+C_L T\beta_{V, W}(Re^{MT})Re^{MT}\Bigr)\le \frac{3R}{4}
		$$
        are fulfilled. Hence for every $R>R_0$ we have
        $$
        \int_0^T\int_{\mathbb{R}^d}h(|u_0|)\mu_t(dx)\,dt\le\gamma R, \quad \int_{\mathbb{R}^d}V(x)\mu_t(dx)\le Re^{Mt}.
        $$
        It follows that the measure
		$$\Pi(dudxdt)=\delta_{u_0}(du)\otimes\mu_t(dx)\,dt$$
		belongs to the set $S_{R}(\sigma)$.
		
		Let us prove (ii).
		Assume that $\Pi\in S_{R}(\sigma)$ and $\mu=\mu_t\,dt$ is a projection of the measure $\Pi$ on $(x, t)$.
		Denote by $\Pi_{x, t}$ the conditional measures for $\Pi$ with respect to the measure~$\mu$.
		As in Remark~\ref{r7} let us consider the Borel function $u\colon\mathbb{R}^d\times[0, T]\to U$
        such that for $\mu$ -- almost all~$(x, t)$ the inequality
		$$
		u(x, t)=\int_{U}u\,\Pi_{x, t}(du)
		$$
        holds. According to Remark~\ref{r7}, the measure $\mu=\mu_t\,dt$ is a solution to the Cauchy problem
		$\partial_t\mu_t=L_{\sigma, u(x, t)}^{*}\mu_t$, $\mu_0=\nu$.
		Hence for every function $\psi\in C^{\infty}_0(\mathbb{R}^d)$ and all $0\le s<t\le T$ one has
        the equality
		$$
		\int_{\mathbb{R}^d}\psi\,d\mu_t-\int_{\mathbb{R}^d}\psi\,d\mu_s=
		\int_s^t\int_{\mathbb{R}^d}L_{\sigma, u(x, \tau)}\psi(x, \tau)\,\mu_{\tau}(dx)\,d\tau.
		$$
		Let the support of $\psi$ be in some open ball $B$.
		According to the condition (H1.1) we have
        $$
        C(B)=\sup_{x\in B, t\in[0, T], \theta\in\mathcal{M}_R(V)}
	\Bigl(\|A(x, t, \mu)\|+|b(x, t, \theta)|+\|Q(x, t, \theta)\|\Bigr)<\infty.
        $$
        Set
        $$
        C=C(B)\sup_{x}\Bigl(\|D^2\psi(x)\|+|\nabla\psi(x)|\Bigr).
        $$
        We have the estimate
		$$
		\Bigl|\int_{\mathbb{R}^d}\psi\,d\mu_t-\int_{\mathbb{R}^d}\psi\,d\mu_s\Bigr|\le
		C|t-s|+C\int_s^t\int_{\mathbb{R}^d}|u(x, \tau)|\,d\mu_{\tau}(dx)\,d\tau.
		$$
		Using Jensen's inequality, we obtain
		$$
		h\Bigl(\frac{1}{|t-s|}\int_s^t\int_{\mathbb{R}^d}|u(x, \tau)|\,d\mu_{\tau}(dx)\,d\tau\Bigr)\le
		\frac{1}{|t-s|}\int_s^t\int_{\mathbb{R}^d}h(|u(x, \tau)|)\,d\mu_{\tau}(dx)\,d\tau.
		$$
		Hence we arrive at the estimate
		$$
		\int_s^t\int_{\mathbb{R}^d}|u(x, \tau)|\,d\mu_{\tau}(dx)\,d\tau\le |t-s|h^{-1}\Bigl(\frac{\gamma R}{|t-s|}\Bigr),
		$$
        where $h^{-1}$ is the inverse function.
		Thus we obtain
		$$
		\Bigl|\int_{\mathbb{R}^d}\psi\,d\mu_t-\int_{\mathbb{R}^d}\psi\,d\mu_s\Bigr|\le
		C|t-s|+C|t-s|h^{-1}\Bigl(\frac{\gamma R}{|t-s|}\Bigr).
		$$
		Since $\lim_{v\to+\infty}h(v)/v=+\infty$, we have $\lim_{v\to 0}vh^{-1}(1/v)=0$.
		Set
		$$\omega_{\psi}(v)=Cv+Cvh^{-1}\Bigl(\frac{\gamma R}{v}\Bigr).$$
		Therefor $\mu\in\mathcal{M}_{R, M}^{\omega}(V)$, where $\omega=\{\omega_{\psi}\}$.
		
	    Let us prove (iii). Assume that $\sigma^n\in\mathcal{M}_{R, M}^{\omega}(V)$,
		$\Pi^n\in S_{R}(\sigma^n)$, the measures $\sigma^n$ converge weakly to a measure $\sigma$ and
		the measures $\Pi^n$ converge weakly to a measure $\Pi$. By Lemma~\ref{lem2} the measure $\sigma$
		belongs to the set $\mathcal{M}_{R, M}^{\omega}(V)$. Note that the projections
		$\mu^n=\mu^n_t\,dt$ of the measures $\Pi^n$ on $(x, t)$ converge weakly to the projection $\mu$ of the measure $\Pi$ on $(x, t)$.
		By Lemma~\ref{lem2} the measure $\mu$ has the form $\mu=\mu_t\,dt$ and belongs to the set $\mathcal{M}_{R, M}^{\omega}(V)$.
		Furthermore, the measures
		$\mu^n$ converge $V$--weakly to the measure~$\mu$. Note that the function $h(|u|)$ is continuous on $U$,
		the measures $\Pi^n$ converge weakly to the measure $\Pi$ and for all $n$ the estimate
		$$
		\int_{U\times\mathbb{R}^d\times[0, T]}h(|u|)\,\Pi^n(dudxdt)\le \gamma R
		$$
		holds. Then this estimate is fulfilled for~$\Pi$ (see Remark~\ref{r0}).
        Thus it suffices to prove that the equality
		$$
		\int_{\mathbb{R}^d}\psi\,d\mu_t=\int_{\mathbb{R}^d}\psi\,d\nu+\int_{U\times\mathbb{R}^d\times[0, t]}L_{\sigma, u}\psi\,d\Pi
		$$
        is fulfilled for every function $\psi\in C_0^{\infty}(\mathbb{R}^d)$ and all $t\in[0, T]$.
		Let us remark that for every $\psi\in C_0^{\infty}(\mathbb{R}^d)$, for all $t\in[0, T]$ and every $n$
		we have
		$$
		\int_{\mathbb{R}^d}\psi\,d\mu_t^n-\int_{\mathbb{R}^d}\psi\,d\nu=
		\int_{U\times\mathbb{R}^d\times[0, t]}L_{\sigma^n, u}\psi\,d\Pi^n.
		$$
        Let the support of $\psi$ be in some open ball $B$.
		Set
		$$
		C_n(t)=\sup_{x\in B}
		\Bigl(\|A(x, t, \sigma^n)-A(x, t, \sigma)\|+|b(x, t, \sigma^n)-b(x, t, \sigma)|+\|Q(x, t, \sigma^n)-Q(x, t, \sigma)\|\Bigr).
		$$
		According to the condition (H1.3), the equality $\lim_{n\to\infty}C_n(t)=0$ holds for all $t\in[0, T]$.
        By the condition (H1.1) we have $\widetilde{C}=\sup_{n, t}C_n(t)<\infty$.
		Hence we derive the estimate
		$$
		\int_{U\times\mathbb{R}^d\times[0, t]}\Bigl|L_{\sigma^n, u}\psi-L_{\sigma, u}\psi\Bigr|\,d\Pi^n\le
		\sup_x\Bigl(|\nabla\psi(x)|+\|D^2\psi(x)\|\Bigr)\int_{U\times\mathbb{R}^d\times[0, t]}C_n(s)(1+|u|)\,d\Pi^n.
		$$
        Since for every number $N\ge 1$ the estimate
        $$
        \int_{U\times\mathbb{R}^d\times[0, t]}C_n(s)(1+|u|)\,d\Pi^n\le (1+N)\int_0^tC_n(t)\,dt+\gamma R\widetilde{C}\sup_{v>N}\frac{v}{h(v)}.
        $$
        is fulfilled and $\displaystyle\lim_{N\to\infty}\sup_{v>N}\frac{v}{h(v)}=0$, we obtain
        $$
        \lim_{n\to\infty}\int_{U\times\mathbb{R}^d\times[0, t]}\Bigl|L_{\sigma^n, u}\psi-L_{\sigma, u}\psi\Bigr|\,d\Pi^n=0.
        $$
		Thus it suffices to verify the equality
		$$
		\lim_{n\to\infty}\int_{U\times\mathbb{R}^d\times[0, t]}L_{\sigma, u}\psi\,d\Pi^n=
		\int_{U\times\mathbb{R}^d\times[0, t]}L_{\sigma, u}\psi\,d\Pi.
		$$
		Note that the function $L_{\sigma, 0}\psi(x, t)$ is bounded and continuous in~$x$.
        Applying Remark~\ref{r4} we obtain
		$$
		\lim_{n\to\infty}\int_0^t\int_{\mathbb{R}^d}L_{\sigma, 0}\psi\,d\mu_t^n\,dt=
		\int_0^t\int_{\mathbb{R}^d}L_{\sigma, 0}\psi\,d\mu_t\,dt.
		$$
		Finally, we need to pass to the limit in the integral of the function
		$\langle Q(x, t, \sigma)u, \nabla\psi(x)\rangle$ with respect to~$\Pi^n$.
		Let $\zeta_N\in C(\mathbb{R}^{d_1})$, $0\le \zeta_N\le 1$, $\zeta_N(u)=1$ if $|u|<N$ and $\zeta_N(u)=0$ if $|u|>2N$.
        Set $u^N=u\zeta_N(u)$.
        Using the inequality $|u-u^N|\le 2|u|I_{|u|>N}$, we obtain
		the estimate
		\begin{multline*}
			\Bigl|\int_{U\times\mathbb{R}^d\times[0, t]}
			\langle Q(x, t, \sigma)(u-u^N), \nabla\psi(x)\rangle\,d\Pi^n\Bigr|\le
			\\
			2\sup_{x, t}\Bigl(\|Q(x, t, \sigma)\||\nabla\psi(x)|\Bigr)\int_{|u|>N, u\in U}|u|\,d\Pi^n\le
			2\gamma R\sup_{x, t}\Bigl(\|Q(x, t, \sigma)\||\nabla\psi(x)|\Bigr)\sup_{v>N}\frac{v}{h(v)}.
		\end{multline*}
		The similar estimate holds for the measure $\Pi$.
		Since the sequence
		$$
		2\gamma R\sup_{x, t}\Bigl(\|Q(x, t, \sigma)\||\nabla\psi(x)|\Bigr)\sup_{v>N}\frac{v}{h(v)}
		$$
		tends to zero as $N\to\infty$, it suffices to pass to the limit in the integral of
		the function $\langle Q(x, t, \sigma)u^N, \nabla\psi(x)\rangle$ with respect to the measure $\Pi^n$.
		Note that the function
        $$\langle Q(x, t, \sigma)u^N, \nabla\psi(x)\rangle$$
		is a bounded and continuous in $(x, u)$ and for every $n$ the projection of $\Pi^n$ on $t$ is Lebesgue measure on $[0, T]$.
		Using Remark~\ref{r4}, we obtain the equality
		\begin{multline*}
			\lim_{n\to\infty}\int_{U\times\mathbb{R}^d\times[0, t]}\langle Q(x, s, \sigma)u^N, \nabla\psi(x)\rangle\Pi^n(dudsdx)=
			\\
			\int_{U\times\mathbb{R}^d\times[0, t]}\langle Q(x, s, \sigma)u^N, \nabla\psi(x)\rangle\Pi(dudsdx).
		\end{multline*}
	\end{proof}

    Note that the function $\omega_{\psi}(v)=Cv+Cvh^{-1}(\gamma R/v)$ from the statement (ii) of Lemma~\ref{lem2}
    does not depend on $\sigma$.

	Below we use the following well-known results on the uniqueness of sub-probability solutions to the Cauchy problem for
    the Fokker--Planck--Kolmogorov equations.	
	
    \begin{remark}\label{r8}\rm
		As in Remark~\ref{r3} we consider the Cauchy problem
		$$
		\partial_t\mu_t=\mathcal{L}^{*}\mu_t, \quad \mu_0=\nu,
		$$
		where $\nu$ is a Borel probability measure on $\mathbb{R}^d$ and
		$$
		\mathcal{L}\psi(x, t)={\rm trace}\bigl(\alpha(x, t)D^2\psi(x)\bigr)+\langle\beta(x, t), \nabla\psi(x)\rangle.
		$$
		Here the matrix $\alpha=(\alpha^{ij})$ is symmetric and nonnegative definite and the functions
		$\alpha^{ij}$, $\beta^i$ are Borel measurable. Suppose that for every open ball
		$B\subset\mathbb{R}^d$ there exist numbers $\Lambda(B)>0$ and $\lambda(B)>0$ such that
		$$
		\alpha(x, t)\ge \lambda(B)I, \quad \|\alpha(x, t)-\alpha(y, t)\|\le \Lambda(B)|x-y| \quad x, y\in B, \, t\in[0, T].
		$$
		Assume also that $\sup_{x\in B, t\in[0, T]}|\beta(x, t)|<\infty$ for every open ball~$B$.
		Suppose that there exists a solution $\mu=\mu_t\,dt$ such that $\mu_t\ge 0$,
        $\mu_t(\mathbb{R}^d)\le 1$ for every $t\in[0, T]$ and
        $$\alpha^{ij}, \beta^i\in L^1(\mathbb{R}^d\times[0, T], \mu).$$
        Then, according to \cite[Theorem 9.4.3]{bookFPK}, the class of sub--probability solutions
        contains precisely one element~$\mu=\mu_t\,dt$.
	\end{remark}
	
In the following lemma we prove that for a given sequence $\sigma^n$ converging to $\sigma$ and some measure $\Pi\in S_R(\sigma)$
there exists a sequence of measures $\Pi^n\in S_R(\sigma^n)$ such that the measures $\Pi^n$ converge weakly to the measure $\Pi$.  	
This assertion plays a crucial role in the proof of Theorem~\ref{th1}. The main difficulty is that we need to solve
degenerate Fokker--Planck--Kolmogorov equations with irregular coefficients.
A similar problem arises in the proof of the Ambrosio--Figalli--Trevisan superposition principle \cite{Superp21}.
Therefor the part of the proof repeats the arguments from \cite{Superp21}.
We use the method of doubling variables described in \cite[Section 3.3]{LebRi}.
Note that this method is just an analytical form of a known method in stochastic differential equations
(see \cite[Lemma 8.1.3]{Strook}, \cite[Chapter 4]{Ikeda-Watanabe} or \cite{Chen}).
	
	\begin{lemma}\label{lem4}
		There exists a number $R_0$ such that for every $R>R_0$ the following statement is fulfilled.
		Suppose that $\sigma^n, \sigma\in\mathcal{M}_{R, M}(V)$
		and the measures $\sigma^n$ converge $V$--weakly to the measure~$\sigma$.
		Then for every measure $\Pi\in S_{R}(\sigma)$ of the form
		$$\Pi(dudxdt)=\delta_{u(x, t)}(du)\mu_t(dx)\,dt,$$
		where $(x, t)\mapsto u(x, t)$ is a Borel function from $\mathbb{R}^d\times[0, T]$ to $U$, there exists
		a sequence of measures $\Pi^n\in S_{R}(\sigma^n)$ such that for every $n$
		the projection of $\Pi^n$ on $(u, t)$ is equal to the projection of $\Pi$ on $(u, t)$
		and the measures $\Pi^n$ converge weakly to the measure $\Pi$.
	\end{lemma}
	\begin{proof}
		The proof is in several steps.
		
		\vspace*{0.1cm}
		
		{\bf I. (The Doubling of variables)}
		
		\vspace*{0.1cm}
		
		Let $(x, y)\in\mathbb{R}^d_x\times\mathbb{R}^d_y$.
		Set
		$$
		\mathbb{A}_n(x, y, t)=\left(
		\begin{array}{cc}
			A(x, t, \sigma^n) & \sqrt{A(x, t, \sigma^n)}\sqrt{A(y, t, \sigma)} \\
			\sqrt{A(y, t, \sigma)}\sqrt{A(x, t, \sigma^n)} & A(y, t, \sigma) \\
		\end{array}
		\right)
		$$
		and
		$$
		\mathbb{B}_n(x, y, t)=\left(
		\begin{array}{l}
			b(x, t, \sigma^n)+Q(x, t, \sigma^n)u(y, t) \\
			b(y, t, \sigma)+Q(y, t, \sigma)u(y, t) \\
		\end{array}
		\right).
		$$
		Let us consider the differential operator
		$$
		\mathbb{L}_n\psi(x, y)={\rm trace}\bigl(\mathbb{A}_n(x, y, t)D^2\psi(x, y)\bigr)+
		\langle\mathbb{B}_n(x, y, t), \nabla\psi(x, y)\rangle.
		$$
		Denote by $\nu\circ(x, x)^{-1}$ the push-forward measure of $\nu$ by the mapping $x\mapsto (x, x)$.
        Suppose that the Cauchy problem
		$$
		\partial_t\pi_t=\mathbb{L}_n^{*}\pi_t, \quad \pi_0=\nu\circ(x, x)^{-1}
		$$
		has a probability solution $\pi^n=\pi^n_t\,dt$ such that 1) the mapping $t\mapsto\pi^n_t$ is continuous,
		2) for every $t$ the projection of the measure $\pi^n_t$ on $y$ is equal to the measure $\mu_t$,
		3) for every $t$ the estimate
		$$
		\int_{\mathbb{R}^d\times\mathbb{R}^d}V(x)\pi_t^n(dxdy)\le Re^{Mt}
		$$
        holds. The steps V--IX are devoted to the existence of the measure $\pi^n$.
		Denote by $\Pi^n$ the projection of the measure
		$$
		\delta_{u(y, t)}(du)\pi_t^n(dxdy)\,dt
		$$
        on the space $U\times\mathbb{R}^d\times[0, T]$ of variables $(u, x, t)$.
		Note that for every Borel bounded function $\varphi$ on $U\times[0, T]$ the equalities
		\begin{multline*}
			\int_{U\times\mathbb{R}^d\times[0, T]}\varphi(u, t)\Pi^n(dudxdt)=\int_{\mathbb{R}^d\times\mathbb{R}^d\times[0, T]}\varphi(u(y, t), t)\pi_t^n(dxdy)\,dt=
			\\
			=\int_{\mathbb{R}^d\times[0, T]}\varphi(u(y,t), t)\mu_t(dy)\,dt=\int_{U\times\mathbb{R}^d\times[0, T]}\varphi(u, t)\Pi(dudxdt)
		\end{multline*}
        are fulfilled. Hence the projection of the measure $\Pi^n$ on $(u, t)$ is equal to the projection of the measure $\Pi$ on $(u, t)$,
        in particular, we have the estimate
		$$
		\int_{U\times\mathbb{R}^d\times[0, T]} h(|u|)\Pi^n(dudxdt)\le \gamma R.
		$$
		Denote by $\mu^n_t(dx)$ the projection of $\pi^n_t$ on $x$. Since for every Borel bounded function
        $\varphi$ on $\mathbb{R}^d\times[0, T]$ the equalities
		$$
		\int_{U\times\mathbb{R}^d\times[0, T]}\varphi(x, t)\Pi^n(dudxdt)=\int_{\mathbb{R}^d\times\mathbb{R}^d\times[0, T]}\varphi(x, t)\pi_t^n(dxdy)\,dt=
		\int_{\mathbb{R}^d\times[0, T]}\varphi(x, t)\mu_t^n(dx)\,dt
		$$
		are fulfilled, the projection of the measure $\Pi^n$ on $(x, t)$ is equal to $\mu^n_t(dx)\,dt$.
		
		According to the condition (H2.3), we have
		$$
		\|A(y, t, \sigma)\|+|b(y, t, \sigma)|+h^{*}(\|Q(y, t, \sigma)\|)\le C_2(\sigma)V(y).
		$$
		Moreover, the inequality
		$$
		\|Q(y, t, \sigma)\||u(y, t)|\le h^{*}(\|Q(y, t, \sigma)\|)+h(|u(y, t)|)
		$$
        holds. Hence for every compact set $K\subset\mathbb{R}^d$ (with respect to~$x$)
        the coefficients of $\mathbb{L}_n$ are integrable with respect to the
		measure~$\pi_t^n\,dt$ on $K\times\mathbb{R}^d\times[0, T]$.
        Let $\zeta\in C_0^{\infty}(\mathbb{R}^d)$, $\eta\in C_0^{\infty}(\mathbb{R}^d)$,
		and $\eta(y)=1$ if $|y|\le 1$.
		Substitute the function $\psi(x, y)=\zeta(x)\eta(y/k)$ in the equality
		$$
		\int_{\mathbb{R}^d\times\mathbb{R}^d}\psi(x, y)\pi_t^n(dxdy)-\int_{\mathbb{R}^d}\psi(x, x)\nu(dx)
		=\int_0^t\int_{\mathbb{R}^d\times\mathbb{R}^d}\mathbb{L}_n\psi(x, y, s)\pi_s^n(dxdy)\,ds.
		$$
		Letting $k\to\infty$, we obtain
		$$
		\int_{\mathbb{R}^d}\zeta(x)\mu_t^n(dx)-\int_{\mathbb{R}^d}\zeta(x)\nu(dx)
		=\int_0^t\int_{\mathbb{R}^d\times\mathbb{R}^d}L_{\sigma^n, u(y, s)}\zeta(x, s)\pi_s^n(dxdy)\,ds.
		$$
		Note that
		$$
		\int_0^t\int_{\mathbb{R}^d\times\mathbb{R}^d}L_{\sigma^n, u(y, s)}\zeta(x, s)\pi_s^n(dxdy)\,ds=
		\int_{U\times\mathbb{R}^d\times[0, t]}L_{\sigma^n, u}\zeta(x, s)\Pi^n(dudxds).
		$$
		Thus the measure $\Pi^n$ belongs to the set $S_{R}(\sigma^n)$.
		
		Let us verify that the measures $\Pi^n$ converge weakly to the measure $\Pi$.
        Since $\Pi^n, \Pi\in P_R$ and $P_R$ is a compact set, it suffices to prove that
        if the measures $\Pi^n$ converge, then the limit is equal to $\Pi$.
        To do this it suffices to find a subsequence
        converging to the measure $\Pi$.

		\vspace*{0.1cm}
		
		{\bf II. (The convergence of a subsequence $\pi^{n_k}=\pi^{n_k}_t\,dt$)}
		
		\vspace*{0.1cm}
		
		Let us show that one can extract a subsequence $n_k$ such that
		for every $t\in[0, T]$ the measures $\pi^{n_k}_t$ converge weakly to a probability measure $\pi_t$,
		solving the Cauchy problem $\partial_t\pi_t=\mathbb{L}^{*}\pi_t$, $\pi_0=\nu\circ(x, x)^{-1}$, with the operator
		$$
		\mathbb{L}\psi(x, y)={\rm trace}\bigl(\mathbb{A}(x, y, t)D^2\psi(x, y)\bigr)+
		\langle\mathbb{B}(x, y, t), \nabla\psi(x, y)\rangle,
		$$
		where
		$$
		\mathbb{A}(x, y, t)=\left(
		\begin{array}{cc}
			A(x, t, \sigma) & \sqrt{A(x, t, \sigma)}\sqrt{A(y, t, \sigma)} \\
			\sqrt{A(y, t, \sigma)}\sqrt{A(x, t, \sigma)} & A(y, t, \sigma) \\
		\end{array}
		\right)
		$$
		and
		$$
		\mathbb{B}(x, y, t)=\left(
		\begin{array}{l}
			b(x, t, \sigma)+Q(x, t, \sigma)u(y, t) \\
			b(y, t, \sigma)+Q(y, t, \sigma)u(y, t) \\
		\end{array}
		\right).
		$$
		Note that for every $n$ the projection of $\pi^n_t$ on $y$ is equal to $\mu_t$ and the estimate
		$$
		\int_{\mathbb{R}^d\times\mathbb{R}^d}V(x)\pi^n_t(dxdy)\le Re^{Mt}
		$$
		holds for all $t\in[0, T]$. By the Prokhorov theorem for every $t$ one can extract from the sequence $\pi^n_t$ a convergent subsequence.
		Let $\psi\in C_0^{\infty}(\mathbb{R}^d\times\mathbb{R}^d)$ and the support of $\psi$ be in $B\times B$,
		where $B$ is an open ball in $\mathbb{R}^d$. Set
		$$
		C(B)=\sup_{x\in B, t\in [0, T], \theta\in\mathcal{M}_{R, M}(V)}\Bigl(\|A(x, t, \theta)\|
		+|b(x, t, \theta)|+\|Q(x, t, \theta)\|\Bigr).
		$$
		Then for every $0\le s<t\le T$ we have
		$$
		\Bigl|\int_s^t\int_{\mathbb{R}^d\times\mathbb{R}^d}\mathbb{L}\psi(x, y, s)\pi^n_s(dxdy)\,ds\Bigr|
		\le 2C(B)C(\psi)|t-s|+2C(B)C(\psi)\int_s^t\int_{\mathbb{R}^d}|u(y, s)|\mu_s(dy)\,ds,
		$$
		where
		$$
		C(\psi)=\sup_{x, y}\Bigl(|\nabla\psi(x, y)|+\|D^2\psi(x, y)\|\Bigr).
		$$
		Arguing as in Lemma~\ref{lem3}, we obtain
		$$
		\Bigl|\int_{\mathbb{R}^d\times\mathbb{R}^d}\psi(x, y)\pi_t^n(dxdy)-\int_{\mathbb{R}^d\times\mathbb{R}^d}\psi(x, y)\pi_s^n(dxdy)\Bigr|\le
		\omega_{\psi}(|t-s|),
		$$
		where $\omega_{\psi}(v)=2C(B)C(\psi)v+2C(B)C(\psi)vh^{-1}(\gamma R/v)$.
        Repeating the arguments from the proof of Lemma~\ref{lem2}, we derive the
		existence of a sequence $n_k$ such that for every $t$
		the measures $\pi^{n_k}_t$ converge weakly to a probability measure $\pi_t$.
		Moreover, the mapping $t\mapsto\pi_t$ is continuous with respect to the weak topology. Note also that for every $t$ the
        projection of $\pi_t$ on $y$ is equal to the measure~$\mu_t$.
		
		Let us verify that $\pi_t\,dt$ is a solution to the Cauchy problem.
		Let $\psi\in C_0^{\infty}(\mathbb{R}^d\times\mathbb{R}^d)$ and the support of
		$\psi$ be in the set $B\times B$, where $B$ is an open ball in $\mathbb{R}^d$.
		We need to pass to the limit in the equality
		$$
		\int_{\mathbb{R}^d\times\mathbb{R}^d}\psi(x, y)\pi_t^{n_k}(dxdy)-\int_{\mathbb{R}^d}\psi(x, x)\nu(dx)
		=\int_0^t\int_{\mathbb{R}^d\times\mathbb{R}^d}\mathbb{L}_{n_k}\psi(x, y, s)\pi_s^{n_k}(dxdy)\,ds.
		$$
		Set
		\begin{multline*}
			C_k(t)=\sup_{x\in B}
			\Bigl(\|A(x, t, \sigma^{n_k})-A(x, t, \sigma)\|+|b(x, t, \sigma^{n_k})-b(x, t, \sigma)|+
			\\
			+\|Q(x, t, \sigma^{n_k})-Q(x, t, \sigma)\|\Bigr)
		\end{multline*}
		and
		$$
		C(\psi)=\sup_{x, y}\Bigl(|\nabla\psi(x, y)|+\|D^2\psi(x, y)\|\Bigr).
		$$
		We have the estimate
		$$
		\Bigl|\mathbb{L}_{n_k}\psi(x, y, t)-\mathbb{L}\psi(x, y, t)\Bigr|\le
		C_k(t)C(\psi)+C_k(t)C(\psi)|u(y, t)|.
		$$
        Note that
        \begin{multline*}
        \int_0^T\int_{\mathbb{R}^d\times\mathbb{R}^d}\Bigl(C_k(t)C(\psi)+C_k(t)C(\psi)|u(y, t)|\Bigr)\pi_t^{n_k}(dxdy)\,dt=
        \\
        =\int_0^T\int_{\mathbb{R}^d}\Bigl(C_k(t)C(\psi)+C_k(t)C(\psi)|u(y, t)|\Bigr)\mu_t(dy)\,dt.
        \end{multline*}		
        Since $\lim_{k\to\infty}C_k(t)=0$ and $\sup_{k, t}C_k(t)<\infty$, we obtain
        $$
        \lim_{k\to\infty}\int_0^T\int_{\mathbb{R}^d}\Bigl(C_k(t)C(\psi)+C_k(t)C(\psi)|u(y, t)|\Bigr)\mu_t(dy)\,dt=0.
        $$
        Hence we have
		$$
		\lim_{k\to\infty}\int_0^t\int_{\mathbb{R}^d\times\mathbb{R}^d}
		\Bigl|\mathbb{L}_{n_k}\psi(x, y, s)-\mathbb{L}\psi(x, y, s)\Bigr|\pi^{n_k}_s(dxdy)\,ds=0.
		$$
		Thus it suffices to prove that
		$$
		\lim_{k\to\infty}\int_0^t\int_{\mathbb{R}^d\times\mathbb{R}^d}\mathbb{L}\psi(x, y, s)\pi^{n_k}_s(dxdy)\,ds=
		\int_0^t\int_{\mathbb{R}^d\times\mathbb{R}^d}\mathbb{L}\psi(x, y, s)\pi_s(dxdy)\,ds.
		$$
		We only consider the terms with
		$$
		\langle Q(x, t, \sigma)u(y, t), \nabla_x\psi(x, y)\rangle \quad \hbox{\rm and} \quad
		\langle Q(y, t, \sigma)u(y, t), \nabla_y\psi(x, y)\rangle.
		$$
		To use the arguments from Remark~\ref{r4} we need to replace the function $(y, t)\mapsto u(y, t)$
		with a bounded continuous function $(y, t)\mapsto\widetilde{u}(y, t)$. Indeed, we have the estimate
		\begin{multline*}
			\Bigl|\int_0^T\int_{\mathbb{R}^d\times\mathbb{R}^d}
			\langle Q(x, t, \sigma)(u(y, t)-\widetilde{u}(y, t)), \nabla_x\psi(x, y)\rangle(\pi_t^n+\pi)(dxdy)\,dt\Bigr|+
\\
\Bigl|\int_0^T\int_{\mathbb{R}^d\times\mathbb{R}^d}
			\langle Q(y, t, \sigma)(u(y, t)-\widetilde{u}(y, t)), \nabla_y\psi(x, y)\rangle(\pi_t^n+\pi)(dxdy)\,dt\Bigr|
\le
			\\
			\le 4C(\psi)\sup_{x\in B, t\in[0, T], \theta\in\mathcal{M}_{R, M}(V)}\|Q(x, t, \theta)\|
			\int_0^T\int_{\mathbb{R}^d}|u(y, t)-\widetilde{u}(y, t)|\mu_t(dy)\,dt.
		\end{multline*}
		By choosing $\widetilde{u}$, one can make the right-hand side arbitrary small.
		
		Finally, we note that for every $t\in[0, T]$ one has the estimate		
        $$
		\int_{\mathbb{R}^d\times\mathbb{R}^d}V(x)\pi_t(dxdy)\le Re^{Mt}.
		$$
		
		\vspace*{0.1cm}
		
		{\bf III. (The equality $\pi_t=\mu_t\circ(y, y)^{-1}$)}.
		
		\vspace*{0.1cm}
		
		We know that the projection of $\pi_t$ on $y$ is equal to $\mu_t$. To prove $\pi_t=\mu_t\circ(y, y)^{-1}$
		it suffices to verify that for every $t$ the equality $x=y$ holds for $\pi_t$ -- almost all $(x, y)$.
		
		Note that for every function $H\in C^2[0, +\infty)$ we have the equality
		\begin{multline*}
			\mathbb{L}H\Bigl(\frac{|x-y|^2}{2}\Bigr)=
			H'\Bigl(\frac{|x-y|^2}{2}\Bigr){\rm trace}\Biggl(\Bigl(\sqrt{A(x, t, \sigma)}-\sqrt{A(y, t, \sigma)}\Bigr)^2\Biggr)+
			\\
			H''\Bigl(\frac{|x-y|^2}{2}\Bigr)\Big\langle(\sqrt{A(x, t, \sigma)}-\sqrt{A(y, t, \sigma)}\bigr)^2(x-y), x-y\Big\rangle+
			\\
			H'\Bigl(\frac{|x-y|^2}{2}\Bigr)
			\Big\langle\bigl(b(x, t, \sigma)-b(y, t, \sigma)\bigr)+
			\bigl(Q(x, t, \sigma)-Q(y, t, \sigma)\bigr)u(y, t), x-y\Big\rangle.
		\end{multline*}
		Let $\delta>0$. Let us consider the function
		$$
		H(v)=\ln\Bigl(1+\frac{v}{\delta}\Bigr).
		$$
		Let us remark that $H'(v)>0$, $H'(v)v\le 1$ and $-1\le H''(v)v^2\le 0$. Hence
        $$
        H''\Bigl(\frac{|x-y|^2}{2}\Bigr)\Big\langle\Bigl(\sqrt{A(x, t, \sigma)}-\sqrt{A(y, t, \sigma)}\Bigr)^2(x-y), x-y\Big\rangle\le 0.
        $$
		Applying the condition (H2.2), we obtain
		\begin{multline*}
			\mathbb{L}\ln\Bigl(1+\frac{|x-y|^2}{2\delta}\Bigr)\le C_1\Bigl(V(x)+V(y)+
			h^{*}(\Theta(x, t))+h^{*}(\Theta(y, t))+2h(|u(t, y)|)\Bigr)\le
			\\
			\le C_2\Bigl(V(x)+V(y)+h(|u(t, y)|)\Bigr),
		\end{multline*}
		where the numbers $C_1$ and $C_2$ do not depend on $\delta$.
		Let $\zeta\in C_0^{\infty}(\mathbb{R}^d\times\mathbb{R}^d)$, $0\le\zeta\le 1$, $\zeta(x, y)=1$ if $|x|^2+|y|^2\le 1$ and
		$\zeta(x)=0$ if $|x|^2+|y|^2\ge 4$. Set $\zeta_j(x, y)=\zeta(x/j, y/j)$.
		Since
		\begin{multline*}
			\mathbb{L}\zeta_j(x, y)H\Bigl(\frac{|x-y|^2}{2}\Bigr)=\zeta_j(x, y)\mathbb{L}H\Bigl(\frac{|x-y|^2}{2}\Bigr)+
			\\
			H\Bigl(\frac{|x-y|^2}{2}\Bigr)\mathbb{L}\zeta_j(x, y)
			+2\Big\langle\mathbb{A}\nabla\zeta_j(x, y), \nabla H\Bigl(\frac{|x-y|^2}{2}\Bigr)\Big\rangle,
		\end{multline*}
		we have
		\begin{multline*}
			\mathbb{L}\zeta_j(x, y)H\Bigl(\frac{|x-y|^2}{2}\Bigr)\le
			C_2\Bigl(V(x)+V(y)+h(|u(t, y)|)\Bigr)+
			\\
			\frac{1}{j}I_{j^2\le |x|^2+|y|^2\le 4j^2}(x, y)C_3\Bigl(V(x)+V(y)+h(|u(y, t)|)\Bigr)H\Bigl(\frac{|x-y|^2}{2}\Bigr)+
			\\
			\frac{1}{j}I_{j^2\le |x|^2+|y|^2\le 4j^2}(x, y)C_4\Bigl(V(x)+V(y)\Bigr)|x-y|H'\Bigl(\frac{|x-y|^2}{2}\Bigr),
		\end{multline*}
		where the numbers $C_3$ and $C_4$ do not depend on~$\delta$ and $j$.
		Note that the function
		$$
		\frac{1}{j}I_{j^2\le |x|^2+|y|^2\le 4j^2}(x, y)H\Bigl(\frac{|x-y|^2}{2}\Bigr)
		$$
		is bounded on $\mathbb{R}^d\times\mathbb{R}^d$. Moreover, the inequality
        $|x-y|H'\Bigl(\frac{|x-y|^2}{2}\Bigr)\le 1/\sqrt{\delta}$ holds for all $x, y\in\mathbb{R}^d$.
		Substitute the function $\zeta_j(x, y)H\Bigl(\frac{|x-y|^2}{2}\Bigr)$
        into the integral identity defining the solution $\pi_t\,dt$.
		Letting $j\to\infty$,
		we obtain the estimate
		$$
		\int_{\mathbb{R}^d\times\mathbb{R}^d}\ln\Bigl(1+\frac{|x-y|^2}{2\delta}\Bigr)\pi_t(dxdy)\le
		C_2\int_0^t\int_{\mathbb{R}^d\times\mathbb{R}^d}\Bigl(V(x)+V(y)+h(|u(y, t)|)\Bigr)\pi_s(dxdy)\,ds.
		$$
		Thus for every $\delta>0$ we have
		$$
		\int_{\mathbb{R}^d\times\mathbb{R}^d}\ln\Bigl(1+\frac{|x-y|^2}{2\delta}\Bigr)\pi_t(dxdy)\le
		C_2\Bigl(\frac{2R}{M}e^{MT}+\gamma R\Bigr).
		$$
		Letting $\delta\to 0+$, we conclude that the equality $x=y$ holds for $\pi_t$ -- almost all $(x, y)$.
        The idea of this proof, in particular, the choice of $H$, is well known (see, for example, \cite{RZ}).

		\vspace*{0.1cm}
		
		{\bf IV. (The convergence $\Pi^{n_k}\to \Pi$)}.
		
		\vspace*{0.1cm}
		
		Assume that the function $(u, x, t)\mapsto \varphi(u, x, t)$ is bounded and 1--Lipschitzian.
		For every continuous bounded function $\widetilde{u}(y, t)$ we have the estimate
		$$
		\int_0^T\int|\varphi(u(y, t), x, t)-\varphi(\widetilde{u}(y, t), x, t)|\pi^n_t(dxdy)\,dt\le
		\int_0^T\int|u(y, t)-\widetilde{u}(y, t)|\mu_t(dy)\,dt.
		$$
        By choosing $\widetilde{u}$, one can make the right-hand side arbitrary small.
		Using this observation and the weak convergence
		$\pi^{n_k}\to \pi$, we obtain the equality
		$$
		\lim_{k\to\infty}\int_0^T\int_{\mathbb{R}^d\times\mathbb{R}^d}\varphi(u(y, t), x, t)\pi_t^{n_k}(dxdy)\,dt
		=\int_0^T\int_{\mathbb{R}^d\times\mathbb{R}^d}\varphi(u(y, t), x, t)\pi_t(dxdy)\,dt.
		$$
		Finally, we note that
		$$
		\int_0^T\int_{\mathbb{R}^d\times\mathbb{R}^d}\varphi(u(y, t), x, t)\pi_t^{n_k}(dxdy)\,dt=
		\int_{U\times\mathbb{R}^d\times[0, T]}\varphi(u, x, t)\Pi^{n_k}(dudxdt)
		$$
		and
		\begin{multline*}
			\int_0^T\int_{\mathbb{R}^d\times\mathbb{R}^d}\varphi(u(y, t), x, t)\pi_t(dxdy)\,dt=
			\int_0^T\int_{\mathbb{R}^d}\varphi(u(y, t), y, t)\mu_t(dy)\,dt=
			\\
			=\int_{U\times\mathbb{R}^d\times[0, T]}\varphi(u, x, t)\Pi(dudxdt).
		\end{multline*}
		
		Thus we only need to construct $\pi^n$.
		
		\vspace*{0.1cm}
		
		{\bf V. (Smoothing of the coefficients)}
		
		\vspace*{0.1cm}
		
		Firstly, we construct a solution on $[0, T']$, where $T'<T$.
		
		Let $0<\varepsilon<\min\{2^{-1}, T-T'\}$ and $\phi$ be the standard Gaussian density on $\mathbb{R}^d$.
		We also use the function
		$\omega_{\varepsilon}(y)=\varepsilon^{-d}\omega(y/\varepsilon)$, where $\omega\in C_0^{\infty}(\mathbb{R}^d)$, $\omega\ge 0$,
		$\|\omega\|_{L^1}=1$ and the support of $\omega$ is in $\{y\colon \, |y|<1\}$.
        Recall that we are given the measures $\mu_t$
        such that
        $$
        \partial_t\mu_t=L^{*}_{\sigma, u(y, t)}\mu_t, \quad \mu_0=\nu,
        $$
        where
        $$
        L_{\sigma, u(y, t)}\psi(y)={\rm trace}\bigl(A(y, t, \sigma)D^2\psi(y)\bigr)+
        \langle b(y, t, \sigma)+Q(y, t, \sigma)u(y, t), \nabla\psi(y)\rangle.
        $$
		Set
		$$
		\mu_t^{\varepsilon}(y):=\varepsilon\phi(y)
		+\frac{1-\varepsilon}{\varepsilon}\int_t^{t+\varepsilon}\int_{\mathbb{R}^d}\omega_{\varepsilon}(y-z)\mu_s(dz)\,ds,
		$$
		$$
		a^{ij}_{\varepsilon}(y, t):=\frac{1-\varepsilon}{\varepsilon \mu_t^{\varepsilon}(y)}
		\int_t^{t+\varepsilon}\int_{\mathbb{R}^d}a^{ij}(z, s, \sigma)\omega_{\varepsilon}(y-z)\mu_s(dz)\,ds,
		$$
		$$
		b^{i}_{\varepsilon}(y, t):=\frac{1-\varepsilon}{\varepsilon \mu_t^{\varepsilon}(y)}
		\int_t^{t+\varepsilon}\int_{\mathbb{R}^d}b^{i}(z, s, \sigma)\omega_{\varepsilon}(y-z)\mu_s(dz)\,ds,
		$$
		$$
		(Qu)^{i}_{\varepsilon}(y, t):=\frac{1-\varepsilon}{\varepsilon \mu_t^{\varepsilon}(y)}
		\int_t^{t+\varepsilon}\int_{\mathbb{R}^d}Q(z, s, \sigma)u(s, z)\omega_{\varepsilon}(y-z)\mu_s(dz)\,ds,
		$$
		$$
		u_{\varepsilon}(y, t):=\frac{1-\varepsilon}{\varepsilon \mu_t^{\varepsilon}(y)}
		\int_t^{t+\varepsilon}\int_{\mathbb{R}^d}u(z, s)\omega_{\varepsilon}(y-z)\mu_s(dz)\,ds,
		$$
		Let us verify that the function $\mu^{\varepsilon}_t(y)$ is a solution to the Cauchy problem
		$$
		\partial_t\mu^{\varepsilon}_t=L_{\varepsilon}^{*}\mu^{\varepsilon}_t, \quad \mu^{\varepsilon}_0=\nu^{\varepsilon},
		$$
		with the operator
		$$
		L_{\varepsilon}\psi={\rm trace}(A_{\varepsilon}D^2\psi)+\langle b_{\varepsilon}+(Qu)_{\varepsilon}, \nabla\psi\rangle
		+\frac{\varepsilon\phi}{\mu_t^{\varepsilon}}\Bigl(\Delta\psi-\langle y, \nabla\psi\rangle\Bigr).
		$$
		and the initial condition
		$$
		\nu^{\varepsilon}(y)=\varepsilon\phi(y)
		+\frac{1-\varepsilon}{\varepsilon}\int_0^{\varepsilon}\int_{\mathbb{R}^d}\omega_{\varepsilon}(y-z)\mu_s(dz)\,ds.
		$$
        Using the continuity of the mapping $t\mapsto\mu_t$, we obtain the equality
        $$
        \partial_t\mu^{\varepsilon}_t(y)=
        \frac{1-\varepsilon}{\varepsilon}\int_{\mathbb{R}^d}\omega_{\varepsilon}(y-z)\mu_{t+\varepsilon}(dz)-
        \frac{1-\varepsilon}{\varepsilon}\int_{\mathbb{R}^d}\omega_{\varepsilon}(y-z)\mu_{t}(dz).
        $$
        Since $\mu_t\,dt$ is a solution to the Cauchy problem with the operator $L_{\sigma, u(y, t)}$, the right--hand side
        is equal to the integral
        $$
        \frac{1-\varepsilon}{\varepsilon}\int_{t}^{t+\varepsilon}\int_{\mathbb{R}^d}L_{\sigma, u(z, s)}\omega_{\varepsilon}(y-z)\mu_{s}(dz)\,ds,
        $$
        where the operator $L_{\sigma, u(z, s)}$ is applied to the function $z\mapsto \omega_{\varepsilon}(y-z)$.
        Note that
        $$
        L_{\sigma, u(z, s)}\omega_{\varepsilon}(y-z)=a^{ij}(z, s, \sigma)\partial_{y_i}\partial_{y_j}\omega_{\varepsilon}(y-z)-
        \bigl(b^i(z, s, \sigma)+q^{im}(z, s, \sigma)u_m(z, s)\bigr)\partial_{y_i}\omega_{\varepsilon}(y-z),
        $$
        where summation over repeated indexes is assumed. It follows that
        \begin{multline*}
        \frac{1-\varepsilon}{\varepsilon}\int_{t}^{t+\varepsilon}\int_{\mathbb{R}^d}L_{\sigma, u(z, s)}\omega_{\varepsilon}(y-z)\mu_{s}(dz)\,ds=
        \\
        =\partial_{y_i}\partial_{y_j}\Bigl(a^{ij}_{\varepsilon}(y, t)\mu_t^{\varepsilon}(y)\Bigr)-
        \partial_{y_i}\Bigl(\bigl(b^i_{\varepsilon}(y, t)+\bigl(Qu\bigr)_{\varepsilon}^i(y, t)\bigr)\mu_t^{\varepsilon}(y)\Bigr).
        \end{multline*}
        Finally, we note that $\Delta\phi(y)+{\rm div}\bigl(y\phi(y)\bigr)=0$. Thus we obtain the equality
        $\partial_t\mu_t^{\varepsilon}=L_{\varepsilon}^{*}\mu^{\varepsilon}_t$.

		Let us remark that the measures $\nu^{\varepsilon}(y)\,dy$ converge weakly to the measure $\nu(dy)$ as $\varepsilon\to 0$
        and also the measures
		$\mu_t^{\varepsilon}(y)\,dy$ converge weakly to the measure~$\mu_t(dy)$ as $\varepsilon\to 0$.
		Note that the functions
		$$
		a^{ij}_{\varepsilon}, \quad b^{i}_{\varepsilon}, \quad (Qu)^{i}_{\varepsilon}, \quad u_{\varepsilon}
		$$
		are integrable over $\mathbb{R}^d\times[0, T']$ with respect to the measure $\mu_t^{\varepsilon}(y)\,dy\,dt$
        and the $L^1$--norms of these functions are majorized by a constant independent on $\varepsilon$.
		Let us prove this observation. Consider the function $(Qu)^{i}_{\varepsilon}$. We have
		$$
		\int_0^{T'}\int_{\mathbb{R}^d}\bigl|(Qu)_{\varepsilon}\bigr|\mu_t^{\varepsilon}\,dy\,dt\le
		\int_0^{T}\int_{\mathbb{R}^d}\|Q(y, t, \sigma^n)\||u(y, t)|\mu_t(dy)\,dt.
		$$
		Note that
		$$
		\|Q(y, t, \sigma^n\||u(y, t)|\le h^{*}(\|Q(y, t, \sigma^n\|)+h(|u(y, t)|),
		$$
		Taking into account the condition (H2.3), we obtain
		$$
		\int_0^{T'}\int_{\mathbb{R}^d}\bigl|(Qu)_{\varepsilon}\bigr|\mu_t^{\varepsilon}\,dy\,dt\le
		\frac{1}{M}C_2(\mu)Re^{MT}+\gamma R.
		$$
		The remaining coefficients are considered in a similar manner.
		
		Let $\zeta_j(x)=\zeta(x/j)$, where $\zeta\in C_0^{\infty}(\mathbb{R}^d)$,
		$0\le\zeta\le 1$ and $\zeta(x)=1$ if $|x|\le 1$. Set
		$$
		\mathbb{A}_n^{\varepsilon, j}(x, y, t)=\left(
		\begin{array}{cc}
			\zeta_j(x)A(x, t, \sigma^n) & \sqrt{\zeta_j(x)A(x, t, \sigma^n)}\sqrt{\alpha_{\varepsilon}(y, t)} \\
			\sqrt{\alpha_{\varepsilon}(y, t)}\sqrt{\zeta_j(x)A(x, t, \sigma^n)} & \alpha_{\varepsilon}(y, t) \\
		\end{array}
		\right),
		$$
		where
		$$
		\alpha_{\varepsilon}(y, t)=A_{\varepsilon}(y, t, \sigma)+\frac{\varepsilon\phi(y)}{\mu_t^{\varepsilon}(y)}I
		$$
		and
		$$
		\mathbb{B}_n^{\varepsilon, j}(x, y, t)=\left(
		\begin{array}{l}
			\zeta_j(x)b(x, t, \sigma^n)+\zeta_j(x)Q(x, t, \sigma^n)u_{\varepsilon}(y, t) \\
			b_{\varepsilon}(y, t, \sigma)+(Qu)_{\varepsilon}(y, t, \sigma)
			-y\displaystyle\frac{\varepsilon \phi(y)}{\mu_t^{\varepsilon}(y)} \\
		\end{array}
		\right).
		$$
		Let us consider the operator
		$$
		\mathbb{L}_n^{\varepsilon, j}\psi(x, y, t)={\rm trace}\bigl(\mathbb{A}_n^{\varepsilon, j}(x, y, t)D^2\psi(x, y)\bigr)+
		\langle\mathbb{B}_n^{\varepsilon, j}(x, y, t), \nabla\psi(x, y)\rangle.
		$$
		Note that the coefficients of this operator are continuous in $(x, y)$ and bounded on $K\times[0, T]$
        for every compact set $K\subset\mathbb{R}^d\times\mathbb{R}^d$.
		
		Let $\pi_0^{\varepsilon}$ be an optimal plan for the measures $\nu$ and $\nu^{\varepsilon}$
        with respect to the cost function
		$$c(x, y)=\min\{|x-y|, 1\}.$$
		It means that $\pi_0^{\varepsilon}$ is a minimizer of the functional
		$$
		\eta\mapsto \int_{\mathbb{R}^d\times\mathbb{R}^d}c(x, y)\eta(dxdy),
		$$
		over all probability measures $\eta$ with projections $\nu$ and $\nu^{\varepsilon}$ on the factors.
		It is known (see, for instance, \cite{BK}) that the minimum of this functional tends to zero as $\varepsilon\to 0$
        because the measures $\pi_0^{\varepsilon}$ converge weakly to the measure $\nu\circ(x, x)^{-1}$.
		
		According to Remark~\ref{r3}, there exists a measure $\pi^{n, \varepsilon, j}=\pi_t^{n, \varepsilon, j}\,dt$
        such that this measure is given by a family of sub--probability measures $\pi_t^{n, \varepsilon, j}$ and
		$\pi^{n, \varepsilon, j}$ is a solution to the Cauchy problem
		with the operator $\mathbb{L}_n^{\varepsilon, j}$ and the initial condition~$\pi_0^{\varepsilon}$.
		
		\vspace*{0.1cm}
		
		{\bf VI. (The projections of $\pi^{n, \varepsilon, j}$)}
		
		\vspace*{0.1cm}
		Let us consider the function $\zeta_N$ as above and arbitrary function $\psi\in C_0^{\infty}(\mathbb{R}^d)$.
		Substitute the function $\zeta_N(x)\psi(y)$ into the integral identity which defines the solution~$\pi^{n, \varepsilon, j}$.
        Letting $N\to\infty$, we obtain that the projection of the measure $\pi^{n, \varepsilon, j}$ on $y$
		is a solution to the Cauchy problem for the Fokker--Planck--Kolmogorov equation with the
        operator $L_{\varepsilon}$ and the initial condition $\nu^{\varepsilon}$.
		Since the coefficients are integrable with respect to the measure $\mu^{\varepsilon}_t(y)\,dy$ and the matrix $A_{\varepsilon}$
		is locally Lipschitzian in $y$ and locally non-degenerate, the class of sub--probability solutions
        contains precisely one element~$\mu^{\varepsilon}=\mu_t^{\varepsilon}(y)\,dy\,dt$ (see Remark~\ref{r8}).
        It follows that the projection of
		$\pi_t^{n, \varepsilon, j}$ on $y$ is equal to the measure $\mu_t^{\varepsilon}(y)\,dy$.
        In particular, $\pi_t^{n, \varepsilon, j}$ is a probability measure.
		
		For every $t\in[0, T']$
		let denote by $\mu^{n, \varepsilon, j}_t$ the projection of the measure $\pi_t^{n, \varepsilon, j}$ on $x$.
		Substitute the function $V(x)\zeta_N(x)\zeta_N(y)$ into the integral identity which defines the solution~$\pi^{n, \varepsilon, j}$.
        Note that the coefficients depending on $x$ vanish outside a ball of sufficiently large radius.
		Moreover, the coefficients depending on $y$ are integrable with respect to the measure $\pi^{n, \varepsilon, j}$
		since the projection of $\pi^{n, \varepsilon, j}$ on $(y, t)$ equals the measure $\mu_t^{\varepsilon}(y)\,dydt$.
        Letting $N\to\infty$, we obtain the inequality
		$$
		\int_{\mathbb{R}^d}V(x)\mu^{n, \varepsilon, j}_t(dx)=
		\int_{\mathbb{R}^d}V(x)\nu(dx)
		+\int_0^t\int_{\mathbb{R}^d\times\mathbb{R}^d}\zeta_j(x)L_{\sigma^n, u_{\varepsilon}(y, s)}V(x, s)\pi^{n, \varepsilon, j}_s(dxdy)\,ds.
		$$
		According to the condition (H2.1), we have the estimate
		$$
		\zeta_j(x)L_{\sigma^n, u_{\varepsilon}(y, s)}V(x, s)\le C_LV(x)+C_LRe^{Ms}+C_L\beta_{V, W}(Re^{MT})Re^{MT}+h(|u_{\varepsilon}(y, s)|),
		$$
        where $\beta_{V, W}$ is defined in Remark~\ref{r1}.
		We need the following version of Jensen's inequality. Assume that $\Phi$ is a convex nonnegative
        function on $[0, +\infty)$ and $\Phi(0)=0$. Let $\xi\ge 0$ be a measurable function on a measurable space
		$(X, \mathcal{X})$ with a sub--probability measure~$\omega$. Then
		$$
		\Phi\Bigl(\int_X\xi\,d\omega\Bigr)\le \int_X\Phi(\xi)\,d\omega.
		$$
		Let us prove this estimate. If $\omega(X)=0$, then the inequality trivially holds since $\Phi(0)=0$.
        Let $\omega(X)>0$.
		The convexity of $\Phi$ and the equality $\Phi(0)=0$ imply the inequality
        $\Phi(\lambda v)\le \lambda\Phi(v)$ for every $v\ge 0$ and $0<\lambda<1$.
		Thus we obtain
		$$
		\Phi\Bigl(\int_X\xi\,d\omega\Bigr)\le \omega(X)\Phi\Bigl(\omega(X)^{-1}\int_X\xi\,d\omega\Bigr)\le \int_X\Phi(\xi)\,d\omega.
		$$
		Recall that $h$ is a convex and increasing function on $[0, +\infty)$ with $h(0)=0$.
        Using these properties of $h$ and Jensen's inequality, we derive the estimate
		$$
		h(|u_{\varepsilon}(y, t)|)\le \frac{1}{\varepsilon\mu_t^{\varepsilon}(y)}
		\int_t^{t+\varepsilon}\int_{\mathbb{R}^d}h(|u(z, s)|)\omega_{\varepsilon}(y-z)\mu_s(dy)\,ds.
		$$
		Hence we have the inequalities
		$$
		\int_0^{T'}\int_{\mathbb{R}^d}h(|u_{\varepsilon}(y, t)|)\mu_t^{\varepsilon}(y)\,dy\,dt
		\le \int_0^T\int_{\mathbb{R}^d}h(|u(z, s)|)\mu_s(dz)\,ds\le \gamma R.
		$$
		Estimating $\zeta_j(x)L_{\sigma^n, u(y, s)}V(x, s)$, we get
		\begin{multline*}
		\int_{\mathbb{R}^d}V(x)\mu^{n, \varepsilon, j}_t(dx)\le
		\int_{\mathbb{R}^d}V(x)\nu(dx)+\gamma R+C_L T\beta_{V, W}(Re^{MT})Re^{MT}
\\
		+C_L\int_0^t\Bigl(\int_{\mathbb{R}^d}V(x)\mu^{n, \varepsilon, j}_s(dx)+Re^{Ms}\Bigr)ds.
		\end{multline*}
		Using Gronwall's inequality, we obtain
		$$
		\int_{\mathbb{R}^d}V(x)\mu^{n, \varepsilon, j}_t(dx)\le e^{C_LT}\Bigl(\|V\|_{L^1(\nu)}+\gamma R
		+C_L T\beta_{V, W}(Re^{MT})Re^{MT}\Bigr)+\frac{C_L}{M-C_L}Re^{Mt}.
		$$
        Note that
        $$
        \gamma e^{C_LT}=\frac{1}{4} \quad \hbox{\rm and} \quad \frac{C_L}{M-C_L}=\frac{1}{4}.
        $$
		There exists a number $R_0>0$ such that for all $R>R_0$ we have
$$
e^{C_LT}\Bigl(\|V\|_{L^1(\nu)}+\gamma R
		+C_L T\beta_{V, W}(Re^{MT})Re^{MT}\Bigr)\le \frac{3R}{4}.
$$
It follows that for every $R>R_0$ the estimate
		$$
		\int_{\mathbb{R}^d} V(x)\mu^{n, \varepsilon, j}_t(dx)\le Re^{Mt}
		$$
        holds for all $t\in[0, T']$.
		
		\vspace*{0.1cm}
		
		{\bf VII. (The limit of $\pi^{n, \varepsilon, j}$ as $j\to\infty$)}
		
		\vspace*{0.1cm}
		
		Taking into account the last estimate and the fact that the projection of $\pi^{n, \varepsilon, j}_t$ on $y$
        does not depend on $j$, we obtain that for every $t$ the sequence $\pi^{n, \varepsilon, j}_t$ is tight.
		Hence for every $t$ one can extract a convergent subsequence $\pi^{n, \varepsilon, j_m}_t$.
		Arguing as in Lemma~\ref{lem3}, we find a subsequence $j_m$ such that
		for every $t$ the measures $\pi^{n, \varepsilon, j_m}_t$ converge weakly to a probability measure
        $\pi^{n, \varepsilon}_t$. Moreover, the mapping $t\mapsto \pi^{n, \varepsilon}_t$ is continuous with respect to the weak topology.
		For every function $\psi\in C_0^{\infty}(\mathbb{R}^d\times\mathbb{R}^d)$ there exists a number
		$m_0$ such that for all $m>m_0$ and for every $(x, y, t)\in\mathbb{R}^d\times\mathbb{R}^d\times[0, T']$
        one has the equality
		$$
		\mathbb{L}_{n}^{\varepsilon, j_m}\psi(x, y, t)=\mathbb{L}_{n}^{\varepsilon}\psi(x, y, t),
		$$
		where
		$$
		\mathbb{L}_n^{\varepsilon}\psi(x, y, t)={\rm trace}\bigl(\mathbb{A}_n^{\varepsilon}(x, y, t)D^2\psi(x, y)\bigr)+
		\langle\mathbb{B}_n^{\varepsilon}(x, y, t), \nabla\psi(x, y)\rangle
		$$
		and
		$$
		\mathbb{A}_n^{\varepsilon}(x, y, t)=\left(
		\begin{array}{cc}
			A(x, t, \sigma^n) & \sqrt{A(x, t, \sigma^n)}\sqrt{\alpha_{\varepsilon}(y, t)} \\
			\sqrt{\alpha_{\varepsilon}(y, t)}\sqrt{A(x, t, \sigma^n)} & \alpha_{\varepsilon}(y, t) \\
		\end{array}
		\right),
		$$
		$$
		\mathbb{B}_n^{\varepsilon}(x, y, t)=\left(
		\begin{array}{l}
			b(x, t, \sigma^n)+Q(x, t, \sigma^n)u_{\varepsilon}(y, t) \\
			b_{\varepsilon}(y, t, \sigma)+(Qu)_{\varepsilon}(y, t, \sigma)
			-y\displaystyle\frac{\varepsilon \phi(y)}{\mu_t^{\varepsilon}(y)} \\
		\end{array}
		\right).
		$$
		Letting $m\to\infty$, we obtain
		$$
		\int_{\mathbb{R}^d\times\mathbb{R}^d}\psi(x, y)\pi_t^{n, \varepsilon}(dxdy)=
		\int_{\mathbb{R}^d\times\mathbb{R}^d}\psi(x, y)\pi_0^{\varepsilon}(dxdy)+
		\int_0^t\int_{\mathbb{R}^d\times\mathbb{R}^d}\mathbb{L}_n^{\varepsilon}\psi(x, y, s)\pi_s^{n, \varepsilon}(dxdy)\,ds.
		$$
		Furthermore, for every $t$ the projection of $\pi_t^{n, \varepsilon}$ on $y$ is equal to $\mu_t^{\varepsilon}(y)\,dy$
		and the estimate
		$$
		\int_{\mathbb{R}^d\times\mathbb{R}^d}V(x)\,\pi_t^{n, \varepsilon}(dxdy)\le Re^{Mt}
		$$
        is fulfilled.
		
		\vspace*{0.1cm}
		
		{\bf VIII. (The limit of $\pi_t^{n, \varepsilon}$ as $\varepsilon\to 0$)}
		
		\vspace*{0.1cm}
		
		The last estimate and the convergence of measures $\mu_t^{\varepsilon}(y)\,dy$ to
		the measure $\mu_t$ as $\varepsilon\to 0$ imply that for every $t$ the sequence $\pi_t^{n, \varepsilon_m}$ is tight.
        It follow that for every $t$ one can find a subsequence $\varepsilon_m\to 0$ such that
		the measures $\pi_t^{n, \varepsilon_m}$ converge weakly to a probability measure.
		Let $\psi\in C_0^{\infty}(\mathbb{R}^d\times\mathbb{R}^d)$ and the support of $\psi$
		be in $B\times B$, where $B$ is an open ball of radius $r$ centered at zero.
        Denote by $B'$ the open ball of the radius $r+1$ centered at zero.
		Set
		$$
		C(B')=\sup_{x\in B', t\in[0, T]}\Bigl(\|A(x, t, \sigma^n)\|+|b(x, t, \sigma^n)|+\|Q(x, t, \sigma^n)\|\Bigr)
		$$
		and
		$$
		C(\psi)=\sup_{x, y}\Bigl(|\nabla\psi(x, y)|+\|D^2\psi(x, y)\|\Bigr).
		$$
		Since $\omega_{\varepsilon}(y)=0$ if $|y|>\varepsilon$,
		for every $y\in B$ we have the estimates
		$$
		\|\alpha_{\varepsilon}(y, t)\|\le 2C(B')+1, \quad |b_{\varepsilon}(y, t)|\le 2C(B'), \quad
		|(Qu)_{\varepsilon}(y, t)|\le C(B')|u|_{\varepsilon}(y, t),
		$$
		where the function $|u|_{\varepsilon}$ is defined by the same formula as the function $u_{\varepsilon}$
		replacing $u$ by $|u|$.
		Note that for a nonnegative number $C(B', \psi)$ depending on $C(B')$ and $\psi$ the estimate
		$$
		\Bigl|\mathbb{L}_n^{\varepsilon}\psi(x, y)\Bigr|\le C(B', \psi)+C(B', \psi)|u|_{\varepsilon}(y, t)
		$$
		holds. Moreover, the inequalities
		$$
		\int_0^{T'}\int_{\mathbb{R}^d}h(|u|_{\varepsilon}(y, s))\mu_s^{\varepsilon}(dy)\,ds\le
		\int_0^T\int_{\mathbb{R}^d}h(|u(y, t)|)\mu_t(dy)\,dt\le \gamma R
		$$
        are fulfilled. Arguing as in Lemma~\ref{lem3}, we obtain the estimate
		$$
		\Bigl|\int_{\mathbb{R}^d}\psi(x, y)\pi_t^{n, \varepsilon}(dxdy)-\int_{\mathbb{R}^d}\psi(x, y)\pi_s^{n, \varepsilon}(dxdy)\Bigr|\le
		\omega_{\psi}(|t-s|),
		$$
		where $s, t\in[0, T']$ and $\omega_{\psi}(v)=C(B', \psi)v+C(B', \psi)vh^{-1}(R/2v)$.
        Repeating the arguments from Lemma~\ref{lem2}, we find a subsequence
		$\varepsilon_m\to 0$ such that
		for every $t$ the measures $\pi^{n, \varepsilon_m}_t$ converge weakly to a probability measure $\pi^{n}_t$.
        Moreover, the mapping $t\mapsto \pi^{n}_t$ is continuous with respect to the weak topology.
		The passage to the limit in the integral identity defining the solution
		$\pi^{n, \varepsilon_m}_t$ is based on the following two observations. First, we note that the functions
		$a^{ij}$, $b^i$, $q^{im}$ are locally bounded and continuous in $x$.
        Second, let us consider a function $v$ such that $v$ is integrable over $\mathbb{R}^d\times[0, T]$
        with respect to the measure $\mu_t(dy)\,dt$. Let $\widetilde{v}$ be a smooth function with a compact support
		on $\mathbb{R}^d\times[0, T]$. Set
		$$
		v_{\varepsilon}(y, t)=\frac{1-\varepsilon}{\varepsilon\mu_t^{\varepsilon}(y)}\int_t^{t+\varepsilon}\int_{\mathbb{R}^d}
		v(z, s)\omega_{\varepsilon}(y-z)\mu_s(dz)\,ds
		$$
		and
		$$
		\widetilde{v}_{\varepsilon}(y, t)=\frac{1-\varepsilon}{\varepsilon\mu_t^{\varepsilon}(y)}\int_t^{t+\varepsilon}\int_{\mathbb{R}^d}
		\widetilde{v}(z, s)\omega_{\varepsilon}(y-z)\mu_s(dz)\,ds.
		$$
		Then we have the inequality
		$$
		\int_0^{T'}\int_{\mathbb{R}^d}\Big|v_{\varepsilon}(y, t)
		-\widetilde{v}_{\varepsilon}(y, t)\Big|\mu_t^{\varepsilon}(y)\,dy\,dt\le
		\int_0^T\int_{\mathbb{R}^d}\Big|v(y, t)-\widetilde{v}(y, t)\Big|\mu_t(dy)\,dt.
		$$
		By choosing $\widetilde{v}$, one can make the right-hand side arbitrary small.
		Furthermore, the functions $\widetilde{v}_{\varepsilon}$ converge uniformly to the function $\widetilde{v}$
		on $\mathbb{R}^d\times[0, T']$ as $\varepsilon\to 0$.
        Thus to prove the passage to the limit as $\varepsilon\to 0$ we replace the functions $a^{ij}$, $b^i$, $q^{im}u$ and $u$ in the expressions for
        $a^{ij}_{\varepsilon}$, $b^i_{\varepsilon}$, $(Qu)_{\varepsilon}^i$, $u_{\varepsilon}$
        by functions $\widetilde{a}^{ij}$, $\widetilde{b}^i$, $\widetilde{q^{im}u}$ and $\widetilde{u}$ with compact supports and
        then we use the uniform convergence of these new expressions $\widetilde{a}^{ij}_{\varepsilon}$, $\widetilde{b}^i_{\varepsilon}$, $(\widetilde{Qu})_{\varepsilon}^i$, $\widetilde{u}_{\varepsilon}$ as $\varepsilon\to 0$.
        Similarly the passage to the limit is proved in \cite{Superp21}.

		\vspace*{0.1cm}
		
		{\bf IX. (Extension to the whole interval $[0, T]$.)}
		
		\vspace*{0.1cm}
		
		Let $\pi^{n, k}=\pi^{n, k}_t\,dt$ be a solution on $[0, T-1/k]\times\mathbb{R}^d$.
        Extend $\pi^{n, k}_t$ on $[0, T]$ by the rule
		$\pi^{n, k}_t=\pi^{n, k}_{T-1/k}$ if $t\in [T-1/k, T]$. Note that this new measure is a solution only on $[0, T-1/k]$.
		Arguing again as in the step (II) and letting $k\to\infty$, we obtain the required solution $\pi^n$ on the whole interval~$[0, T]$.
	\end{proof}
	
    Below we always assume that $\omega=\{\omega_{\psi}\}$ is constructed in the statement (iii) of Lemma~\ref{lem3}.
	Let denote by $P_{R}^{\omega}$ the set of all measures $\Pi\in P_{R}$ such that
	the projection of $\Pi$ on $(x, t)$ belongs to the set $\mathcal{M}_{R, M}^{\omega}(V)$.
	Since $P_R$ and $\mathcal{M}_{R, M}^{\omega}(V)$ are compact sets in the weak topology,
    the set $P_{R}^{\omega}$ is compact in the weak topology.
	
	Let us consider the mapping
	$(\sigma, \Pi)\mapsto \mathcal{F}_{\sigma}(\Pi)$
    from $\mathcal{M}_{R, M}^{\omega}(V)\times P_{R}^{\omega}$ to $\mathbb{R}$ given by the formula
	$$
	\mathcal{F}_{\sigma}(\Pi)=
	\int_{U\times\mathbb{R}^d\times[0, T]}f(u, x, t, \sigma)\Pi(dudxdt)+
	\int_{\mathbb{R}^d}g(x, \sigma)\,\mu_{T}(dx),
	$$
    where $\mu_t(dx)\,dt$ is the projection of $\Pi$ on $(x, t)$.
	
	\begin{lemma}\label{lem5}
		{\rm (i)} Let $\sigma^n, \sigma\in\mathcal{M}_{R, M}^{\omega}(V)$, $\Pi^n\in S_{R}(\sigma^n)$,
		$\Pi\in S_{R}(\sigma)$. Assume that the measures~$\sigma^n$ converge weakly to the measure~$\sigma$
		and the measures~$\Pi^n$ converge weakly to the measure~$\Pi$. Then
		$$
		\lim\inf_{n\to\infty}\mathcal{F}_{\sigma_n}(\Pi^n)\ge \mathcal{F}_{\sigma}(\Pi).
		$$
		
		{\rm (ii)} Let $\sigma^n, \sigma\in\mathcal{M}_{R, M}^{\omega}(V)$, $\Pi^n\in S_{R}(\sigma^n)$,
		$\Pi\in S_{R}(\sigma)$ and for every $n$ the projection of $\Pi^n$ on~$(u, t)$ is equal to the
        projection of $\Pi$ on~$(u, t)$. Assume that the measures~$\sigma^n$ converge weakly to the measure~$\sigma$
		and the measures~$\Pi^n$ converge weakly to the measure~$\Pi$. Then		
        $$
		\lim_{n\to\infty}\mathcal{F}_{\sigma_n}(\Pi^n)=\mathcal{F}_{\sigma}(\Pi).
		$$
	\end{lemma}
	\begin{proof}
		Let us prove~(i). Denote by $\mu_t^n(dx)\,dt$ the projection of $\Pi^n$ on $(x, t)$
        and by $\mu_t(dx)\,dt$ the projection of $\Pi$ on $(x, t)$.
		First, we prove that
		$$
		\lim_{n\to\infty}\int_{\mathbb{R}^d}g(x, \sigma^n)\,d\mu_T^n=\int_{\mathbb{R}^d}g(x, \sigma)\,d\mu_T.
		$$
		According to the condition~(H3.1), we have
		$$
		|g(x, \sigma_n)|\le C_g\bigl(W(x)+Re^{MT}\bigr), \quad
		|g(x, \sigma)|\le C_g\bigl(W(x)+Re^{MT}\bigr).
		$$
		Let $\varepsilon>0$. There exists a number $m_{\varepsilon}>0$ such that
		the estimate $C_g\bigl(W(x)+Re^{MT}\bigr)\le\varepsilon V(x)$ holds for all $x$ satisfying the inequality $V(x)>m_{\varepsilon}$.
        Then we obtain the estimates
		$$
		\Bigl|\int_{\mathbb{R}^d}g(x, \sigma^n)\,d\mu_T^n-\int_{\mathbb{R}^d}g(x, \sigma)\,d\mu_T^n\Bigr|\le
		\sup_{\{x\colon\, V(x)\le m_{\varepsilon}\}}|g(x, \sigma_n)-g(x, \sigma)|+
		2\varepsilon Re^{MT}.
		$$
		By (H3.3) the sequence
		$$
		\sup_{\{x\colon\, V(x)\le m_{\varepsilon}\}}|g(x, \sigma_n)-g(x, \sigma)|
		$$
		tends to zero as $n\to\infty$.
		Furthermore, arguing as in Remark~\ref{r4}, we obtain the equality
		$$
		\lim_{n\to\infty}\int_{\mathbb{R}^d}g(x, \sigma)\,d\mu_T^n=\int_{\mathbb{R}^d}g(x, \sigma)\,d\mu_T.
		$$
		Thus there exists a number $n_0$ such that for all $n>n_0$ one has the estimate
		$$
		\Bigl|\int_{\mathbb{R}^d}g(x, \sigma^n)\,d\mu_T^n-\int_{\mathbb{R}^d}g(x, \sigma)\,d\mu_T\Bigr|\le
        \varepsilon\bigl(2+2Re^{MT}\bigr).
		$$
		
		Let us prove the passage to the limit in the integral of the function $f$.
		According to the condition (H3.2), the estimate
		$$
		f(u, x, t, \theta)+C_fW(x)+C_fRe^{MT}\ge 0
		$$
        is fulfilled for every measure $\theta\in \mathcal{M}_{R, M}^{\omega}(V)$.
		For a natural number $N$ we set
		$$
		\widetilde{f}_N(u, x, t, \theta)=\min\{f(u, x, t, \theta)+C_fW(x)+C_fRe^{MT}, N\}
		$$
		and
		$$
		f_N(u, x, t, \theta)=\widetilde{f}_N(u, x, t, \theta)-C_fW(x)-C_fRe^{MT}.
		$$
		Note that $f_N\le f$ and $|\widetilde{f}_N(u, x, t, \theta)|\le N$.
		Applying Remark~\ref{r4} we obtain
		$$
		\lim_{n\to\infty}\int_{U\times\mathbb{R}^d\times[0, T]}\bigl(C_fW(x)+C_fRe^{MT}\bigr)\Pi^n(dudxdt)=
		\int_{U\times\mathbb{R}^d\times[0, T]}\bigl(C_fW(x)+C_fRe^{MT}\bigr)\Pi(dudxdt).
		$$
		Let us consider the integral of $\widetilde{f}_N$.
		Let $\varepsilon>0$.
		There exists a compact set $K\subset U\times\mathbb{R}^d$ such that
		$$\Pi^n(K\times[0, T])\ge 1-\varepsilon.$$
        Set
        $$
        C_n(t)=\sup_{(u, x)\in K}\bigl|\widetilde{f}_N(u, x, t, \sigma^n)-\widetilde{f}_N(u, x, t, \sigma)\bigr|.
        $$
		Note that the function $v\mapsto\min\{v, N\}$ is $1$--Lipschitzian.
        By (H3.3) we have the equality
		$\lim_{n\to\infty}C_n(t)=0$ for all $t\in[0, T]$. Moreover, $|C_n(t)|\le 2N$ for all $t\in[0, T]$.
        Since the inequality
        \begin{multline*}
        \Bigl|\int_{U\times\mathbb{R}^d\times[0, T]}\widetilde{f}_N(u, x, t, \sigma^n)\Pi^n(dudxdt)
		-\int_{U\times\mathbb{R}^d\times[0, T]}\widetilde{f}_N(u, x, t, \sigma)\Pi^n(dudxdt)\Bigr|\le
\\
        \le 2\varepsilon N+\int_0^TC_n(t)\,dt
        \end{multline*}
        holds, there exists a number $n_0$ such that for all $n>n_0$ one has the estimate
		$$
		\Bigl|\int_{U\times\mathbb{R}^d\times[0, T]}\widetilde{f}_N(u, x, t, \sigma^n)\Pi^n(dudxdt)
		-\int_{U\times\mathbb{R}^d\times[0, T]}\widetilde{f}_N(u, x, t, \sigma)\Pi^n(dudxdt)\Bigr|\le \varepsilon\bigl(2N+1\bigr).
		$$
		We stress that the function $f_N(u, x, t, \sigma)$ is continuous in $(u, x)$ and bounded.
        Furthermore, for every $n$ the projection of $\Pi^n$ on $t$ is equal to Lebesgue measure on $[0, T]$.
        According to Remark~\ref{r4}, we arrive at the inequality
		$$
		\lim_{n\to\infty}\int_{U\times\mathbb{R}^d\times[0, T]}\widetilde{f}_N(u, x, t, \sigma)\Pi^n(dudxdt)=
		\int_{U\times\mathbb{R}^d\times[0, T]}\widetilde{f}_N(u, x, t, \sigma)\Pi(dudxdt).
		$$
		By choosing $n_0$ large enough we have for all $n>n_0$ the estimate
		$$
		\Bigl|\int_{U\times[0, T]\times\mathbb{R}^d}\widetilde{f}_N(u, x, t, \sigma^n)\Pi^n(dudtdx)
		-\int_{U\times[0, T]\times\mathbb{R}^d}\widetilde{f}_N(u, x, t, \sigma)\Pi(dudtdx)\Bigr|\le\varepsilon\bigl(2N+2\bigr)
		$$
		Thus for every $N$ we have
		$$
		\lim_{n\to\infty}\int_{U\times\mathbb{R}^d\times[0, T]}\widetilde{f}_N(u, x, t, \sigma^n)\Pi^n(dudxdt)=
		\int_{U\times\mathbb{R}^d\times[0, T]}\widetilde{f}_N(u, x, t, \sigma)\Pi(dudxdt).
		$$
		Taking into account the estimate $f_N\le f$, we obtain the inequality
		$$
		\int_{U\times\mathbb{R}^d\times[0, T]}\bigl(\widetilde{f}_N(u, x, t, \sigma)-C_fW(x)-C_fRe^{MT}\bigr)\Pi(dudtdx)
		+\int_{\mathbb{R}^d}g(x, \sigma)\,d\mu_T\le \lim\inf_{n\to\infty}\mathcal{F}_{\sigma_n}(\Pi^n).
		$$
		Note that the equality
        $$
        \lim_{N\to\infty}\widetilde{f}_N(u, x, t, \sigma)=
		f(u, x, t, \sigma)+C_fW(x)+C_fRe^{MT}
		$$
        holds for every $(u, x, t)\in U\times\mathbb{R}^d\times[0, T]$. Applying Fatou's lemma, we derive the estimate
		$$
		\int_{U\times\mathbb{R}^d\times[0, T]}f(u, x, t, \sigma)\Pi(dudtdx)
		+\int_{\mathbb{R}^d}g(x, \sigma)\,d\mu_T\le \lim\inf_{n\to\infty}\mathcal{F}_{\sigma_n}(\Pi^n).
		$$
		This completes the proof of the assertion~(i).
		
		Let us prove~(ii).
		In (i) it is proved that
        $$
		\lim_{n\to\infty}\int_{\mathbb{R}^d}g(x, \sigma^n)\,d\mu_T^n=\int_{\mathbb{R}^d}g(x, \sigma)\,d\mu_T.
		$$
        Hence it suffices to consider only the term with $f$.
        Denote by $\Lambda(dudt)$ the projection of the measures $\Pi^n$ and $\Pi$ on $(u, t)$.
		By (H3.2) for all $(u, x, t)\in U\times\mathbb{R}^d\times[0, T]$ and for every measure $\theta\in\mathcal{M}_{R, M}(V)$
		one has the estimate
		$$
		|f(u, x, t, \theta)|\le C_hh(|u|)+C_fW(x)+C_fRe^{MT}.
		$$
		Let $\zeta$ be a continuous function on $\mathbb{R}^{d_1}$ such that $0\le\zeta\le 1$,
		$\zeta(u)=1$ if $|u|<1$ and $\zeta(u)=0$ if $|u|>2$. Set $\zeta_N(u)=\zeta(u/N)$.
		For $\varepsilon>0$ there exists a number $m_{\varepsilon}>0$ such that the estimate
        $C_fW(x)+C_fRe^{MT}\le\varepsilon V(x)$
		holds for all $x$ satisfying the inequality $V(x)>m_{\varepsilon}$.
		We have
		\begin{multline*}
		\int_{U\times[0, T]\times\mathbb{R}^d}(1-\zeta_N(u))|f(u, x, t, \sigma^n)|\,\Pi^n(dudtdx)\le
\\
		\int_{|u|>N}\Bigl(C_hh(|u|)+C_fm_{\varepsilon}+C_fRe^{MT}\Bigr)\Lambda(dudt)+\varepsilon TRe^{MT}.
		\end{multline*}
        The same estimate holds for $\Pi$.
		Take a number $N$ such that
		$$
		\int_{|u|>N}\Bigl(C_hh(|u|)+C_fm_{\varepsilon}+C_fRe^{MT}\Bigr)\Lambda(dudt)\le \varepsilon.
		$$
		Note that there exists a number $C_N>0$ such that the inequality
		$$
		\zeta_N(u)|f(u, x, t, \sigma^n)|\le C_N+C_NW(x)
		$$
        holds for all $n$ and for all $(u, x, t)\in U\times\mathbb{R}^d\times[0, T]$.
		Let $\eta$ be a continuous function on $\mathbb{R}^{d}$ such that $0\le\eta\le 1$,
		$\eta(x)=1$ if $|x|<1$ and $\eta(x)=0$ if $|x|>2$. Set $\eta_k(x)=\eta(x/k)$.
        There exists a number $\widetilde{m}_{\varepsilon}>0$ such that the estimate
        $C_N+C_NW(x)\le\varepsilon V(x)$
		holds for all $x$ satisfying the inequality $V(x)>\widetilde{m}_{\varepsilon}$.
		We have
		\begin{multline*}
			\int_{U\times\mathbb{R}^d\times[0, T]}(1-\eta_k(x))\zeta_N(u)|f(u, x, t, \sigma^n)|\,\Pi^n(dudxdt)\le
			\\
			\le (C_N\widetilde{m}_{\varepsilon}+C_N)\Pi^n\bigl(\{(u, x, t)\colon |x|>k\}\bigr)+\varepsilon TRe^{MT}.
		\end{multline*}
        The same estimate holds for $\Pi$.
		Take a number $k$ such that the estimate
		$$
		(C_N\widetilde{m}_{\varepsilon}+C_N)(\Pi^n+\Pi)\bigl(\{(u, x, t)\colon |x|>k\}\bigr)<\varepsilon
		$$
        holds for all $n$.
		Thus for all $n$ one has
		\begin{multline*}
		\Bigl|\int_{U\times\mathbb{R}^d\times[0, T]}f(u, x, t, \sigma^n)\,\Pi^n(dudxdt)-
\\
		\int_{U\times\mathbb{R}^d\times[0, T]}\eta_k(x)\zeta_N(u)f(u, x, t, \sigma^n)\,\Pi^n(dudxdt)\Bigr|
		\le 2\bigl(\varepsilon+\varepsilon TRe^{MT}\bigr).
		\end{multline*}
		The same estimate holds for $\Pi$. Therefor it suffices to prove the passage to the limit for the function
        $\eta_k(x)\zeta_N(u)f(u, x, t, \sigma)$ instead of $f(u, x, t, \sigma)$.
        Set
        $$
        \widetilde{C}_n(t)=\sup_{x, u}\Bigl(\eta_k(x)\zeta_N(u)\bigl|f(u, x, t, \sigma^n)-f(u, x, t, \sigma)\bigr|\Bigr).
		$$
        Let us remark that $\sup_{n, t}\widetilde{C}_n(t)<\infty$ and by (H3.3)
        the equality $\lim_{n\to\infty}\widetilde{C}_n(t)=0$ is fulfilled for
        every~$t\in[0, T]$. Note that the inequality
        \begin{multline*}
		\Bigl|\int_{U\times\mathbb{R}^d\times[0, T]}\eta_k(x)\zeta_N(u)f(u, x, t, \sigma^n)\,\Pi^n(dudxdt)
\\
-\int_{U\times\mathbb{R}^d\times[0, T]}\eta_k(x)\zeta_N(u)f(u, x, t, \sigma)\,\Pi^n(dudxdt)\Bigr|\le \int_0^TC_n(t)\,dt
        \end{multline*}
		is fulfilled and the right-hand side tends to zero as $n\to\infty$. Hence it suffices to verify the equality
		\begin{multline*}
			\lim_{n\to\infty}\int_{U\times\mathbb{R}^d\times[0, T]}\eta_k(x)\zeta_N(u)f(u, x, t, \sigma)\,\Pi^n(dudxdt)=
			\\
			\int_{U\times\mathbb{R}^d\times[0, T]}\eta_k(x)\zeta_N(u)f(u, x, t, \sigma)\,\Pi(dudxdt).
		\end{multline*}
        This equality follows from Remark~\ref{r4}, since the function $\eta_k(x)\zeta_N(u)f(u, x, t, \sigma)$
        is continuous in $(u, x)$ and bounded.
	\end{proof}
	
	\vspace*{0.2cm}
	
	\section{\sc Proof of Theorem~\ref{th1} and Corollary~\ref{cor1}}
	
	\vspace*{0.2cm}
	
    Recall that
    $$
    M=5C_L, \quad \gamma=\frac{1}{4}e^{-C_LT},
    $$
    where $C_L$ is the constant from (H2.1).

	Let us prove Theorem~2.1.
	
	\begin{proof}
        The proof is in several steps.

        \vspace*{0.1cm}

        {\bf I. (A priory estimates)}

        \vspace*{0.1cm}

		Let $\sigma\in\mathcal{M}_{R, M}(V)$ and $u_0\in U$. Note that
        $$
        L_{\sigma, u_0}V=L_{\sigma}V+\langle Qu_0, \nabla V\rangle\le L_{\sigma}V+h^{*}(|Q^{\top}\nabla V|)+h(|u_0|).
        $$
		According to the condition (H2.1), we have the estimate
		$$
		L_{\sigma, u_0}V(x)\le C_LV(x)+C_LRe^{Mt}+C_L\beta_{V, W}(Re^{MT})Re^{MT}+h(|u_0|),
		$$
        where $\beta_{V, W}$ is defined in Remark~\ref{r1}.
		By Remark~\ref{r3} there exists a probability solution $\mu_t\,dt$ to the Cauchy problem
		$$
		\partial_t\mu_t=L_{\sigma, u_0}^{*}\mu_t, \quad \mu_0=\nu.
		$$
        Applying the estimate from Remark~\ref{r3} with $C=C_L$, $\mathcal{V}=V$ and
        $$
        \mathcal{W}=C_LRe^{Mt}+C_L\beta_{V, W}(Re^{MT})Re^{MT}+h(|u_0|),
        $$
		we obtain for every $t\in[0, T]$ the estimate
		$$
		\int_{\mathbb{R}^d} V\,d\mu_t\le e^{C_LT}\Bigl(\|V\|_{L^1(\nu)}+C_LT\beta_{V, W}(Re^{MT})Re^{MT}+Th(|u_0|)\Bigr)+\frac{1}{4}Re^{Mt}.
		$$
        Here we also use the equality $M=5C_L$.
		By Remark~\ref{r1} we have $\lim_{R\to\infty}\beta_{V, W}(R)=0$.
        It follows that there exists a number $R_1>0$ such that for every $R>R_1$
		the estimate
        $$
        e^{C_LT}\Bigl(\|V\|_{L^1(\nu)}+C_LT\beta_{V, W}(Re^{MT})Re^{MT}+Th(|u_0|)\Bigr)\le \frac{3R}{4}
        $$
        holds. Let $R>R_1$. Then the inequality
		$$
		\int_{\mathbb{R}^d} V\,d\mu_t\le Re^{Mt}
		$$
        is fulfilled for every $t\in[0, T]$.
		Set
		$$
		F_{\sigma}(u, \mu)=\int_{0}^T\int_{\mathbb{R}^d} f(u(x, t), x, t, \sigma)\,\mu_t(dx)\,dt+\int_{\mathbb{R}^d}g(x, \sigma)\,\mu_T(dx).
		$$
        Using the inequality
        $$
        \sup_{t\in[0, T]}\int_{\mathbb{R}^d}W(x)\,\mu_t(dx)+\sup_{t\in [0, T]}\int_{\mathbb{R}^d}W(x)\,\sigma_t(dx)\le 2\beta(Re^{MT})Re^{MT}
        $$
		and applying the conditions (H3.1) and (H3.2), we obtain the estimate
		$$
		F_{\sigma}(u_0, \mu)\le TC_hh(|u_0|)+2(C_fT+C_g)\beta_{V, W}(Re^{MT})Re^{MT}.
		$$
		Thus it suffices to minimize the functional $(u, \mu)\mapsto F_{\sigma}(u, \mu)$
        only over the set of pairs $(u, \mu)$ such that $\mu_t\,dt$ is a solution to
        the Cauchy problem $\partial_t\mu_t=L^{*}_{\sigma, u(x, t)}\mu_t$, $\mu_0=\nu$, and
		$$
		F_{\sigma}(u, \mu)\le TC_hh(|u_0|)+2(C_fT+C_g)\beta_{V, W}(Re^{MT})Re^{MT}.
		$$
		Note that the right--hand side has the form $\alpha(Re^{MT})Re^{MT}$, where $\lim_{R\to\infty}\alpha(R)=0$.
		By Lemma~\ref{lem1} there exists a number $R_2>R_1$ such that for all $R>R_2$ and all $t\in[0, T]$ the inequalities
		$$
		\int_{\mathbb{R}^d} V\,d\mu_t\le Re^{Mt}, \quad
		\int_0^{T}\int_{\mathbb{R}^d}h(|u(x, t)|)\,d\mu_t\,dt\le\gamma R
		$$
		are fulfilled for every Borel mapping $u\colon\mathbb{R}^d\times[0, T]\to U$ and
		for every measure $\mu\in\mathcal{M}(V)$ satisfying the following conditions:
		{\rm 1)} the measure $\mu=\mu_t\,dt$ is a solution to the Cauchy problem $\partial_t\mu=L^{*}_{\sigma, u(x, t)}\mu$, $\mu_0=\nu$,
		{\rm 2)} the estimate $F_{\sigma}(u, \mu)\le \alpha(Re^{MT})Re^{MT}$ holds.
        Let $R>R_2$.
		Taking into account Remark~\ref{rem-int},
        we conclude that it suffices to minimize the functional $(u, \mu)\mapsto F_{\sigma}(u, \mu)$ only
		over the set of pairs $(u, \mu)$ such that $\mu_t\,dt$ is a solution to
        the Cauchy problem $\partial_t\mu_t=L^{*}_{\sigma, u(x, t)}\mu_t$, $\mu_0=\nu$, and
		$$
		\int V(x)\,d\mu_t\le Re^{Mt}, \quad \int_0^T\int h(|u(x, t)|)\mu_t(dx)\,dt\le \gamma R.
		$$
		
        \vspace*{0.1cm}

        {\bf II. (Relaxed control)}

        \vspace*{0.1cm}

        Assume that $u\colon\mathbb{R}^d\times[0, T]\to U$ is a Borel function and the measure $\mu=\mu_t\,dt$
		is given by a family of probability measures $(\mu_t)_{t\in[0, T]}$ such that
        the mapping $t\mapsto\mu_t$ is continuous and the above estimates are fulfilled.
		Using this pair $(u, \mu)$, one can define the measure
		$$
		\Pi(dudxdt)=\delta_{u(x, t)}(du)\mu_t(dx)\,dt
		$$
		on $U\times\mathbb{R}^d\times[0, T]$. If $\mu$ is a solution to the Cauchy problem
		$\partial_t\mu_t=L_{\sigma, u(x, t)}^{*}\mu_t$, $\mu_0=\nu$,
        then $\Pi$ belongs to the set $S_R(\sigma)$ (see the definition before Remark~\ref{r7}).
		
		Let us consider an arbitrary measure $\Pi\in S_R(\sigma)$.
		Let $\mu=\mu_t\,dt$ be a projection of $\Pi$ on~$(x, t)$.
        Denote by $\Pi_{x, t}(du)$ the conditional measures for $\Pi$.
		By Remark~\ref{r7} there exists a Borel function
		$(x, t)\mapsto u(x, t)$ such that the equality
		$$
		u(x, t)=\int_Uu\,\Pi_{x, t}(du)
		$$
        holds for $\mu$ -- almost all $(x, t)$. Set
		$$
		\mathcal{F}_{\sigma}(\Pi)=
		\int_{U\times\mathbb{R}^d\times[0, T]}f(u, x, t, \sigma)\Pi(dudxdt)+
		\int_{\mathbb{R}^d}g(x, \sigma)\,\mu_{T}(dx).
		$$
		Recall that the function $f$ is convex in $u$. Applying the Jensen's inequality,
		we obtain
		\begin{multline*}
			\int_{U\times\mathbb{R}^d\times[0, T]}f(u, x, t, \sigma)\Pi(dudxdt)=
			\int_0^T\int_{\mathbb{R}^d}\int_Uf(u, x, t, \sigma)\Pi_{x, t}(du)\mu_t(dx)\,dt\ge
			\\
			\ge\int_0^T\int_{\mathbb{R}^d}f(u(x, t), x, t, \sigma)\mu_t(dx)\,dt.
		\end{multline*}
		Hence we have the inequality
		$\mathcal{F}_{\sigma}(\Pi)\ge \mathcal{F}_{\sigma}(\widetilde{\Pi})$,
		where
		$$
		\widetilde{\Pi}(dudxdt)=\delta_{u(x, t)}(du)\mu_t(dx)\,dt.
		$$
		Using Jensen's inequality again, we obtain
		$$
		\int_{U\times\mathbb{R}^d\times[0, T]}h(|u|)\Pi(dudxdt)\ge
		\int_{U\times\mathbb{R}^d\times[0, T]}h(|u(x, t)|)\mu_t(dx)\,dt=
		\int_{U\times\mathbb{R}^d\times[0, T]}h(|u|)\widetilde{\Pi}(dudxdt).
		$$
		Since for every function $\psi\in C_0^{\infty}(\mathbb{R}^d)$ and every $t\in[0, T]$ we have
		\begin{multline*}
			\int_{U\times\mathbb{R}^d\times[0, t]}L_{\sigma, u}\psi(x)\Pi(dudxds)=
			\int_{\mathbb{R}^d\times[0, t]}L_{\sigma, u(x, s)}\psi(x)\mu_s(dx)\,ds=
			\\
			=\int_{U\times\mathbb{R}^d\times[0, t]}L_{\sigma, u}\psi(x)\widetilde{\Pi}(dudxds),
		\end{multline*}
		the measure $\widetilde{\Pi}$ belongs to the set $S_R(\sigma)$.
        Thus we can minimize the functional $\mathcal{F}_{\sigma}(\Pi)$ on $S_R(\sigma)$ instead of $F_{\sigma}(u, \mu)$.
		Let us take a number $R>R_2$ such that for $S_R(\sigma)$ the statements of Lemma~\ref{lem3} and Lemma~\ref{lem4} hold.

        Denote by $M_R(\sigma)$ the set of all minimizers of the functional $\mathcal{F}_{\sigma}(\Pi)$ over $S_R(\sigma)$.		
        By Lemma~\ref{lem5} the mapping $\Pi\mapsto \mathcal{F}_{\sigma}(\Pi)$ is lower semi-continuous. By Lemma~\ref{lem3} the
        set $S_R(\sigma)$ is compact. Thus $M_R(\sigma)$ is compact and contains at least one element.
        Moreover, the set $M_R(\sigma)$ is convex since the mapping $\Pi\mapsto \mathcal{F}_{\sigma}(\Pi)$ is linear.

        \vspace*{0.1cm}

        {\bf III. (Fixed point)}

        \vspace*{0.1cm}

        According to Lemma~\ref{lem3}, for every $\Pi\in S_R(\sigma)$ the projection of $\Pi$ on $(x, t)$ belongs to the set
        $\mathcal{M}_{R, M}^{\omega}(V)$, where $\omega=\{\omega_{\psi}\}$ does not depend on $\sigma$.
		To prove Theorem~\ref{th1} it suffices to find a measure
		$\mu\in\mathcal{M}_{R, M}^{\omega}(V)$ such that there exists a minimizer $\Pi$ of
		the functional $\mathcal{F}_{\mu}(\Pi)$ on the set $S_{R}(\mu)$ and the projection of $\Pi$ on $(x, t)$
        is equal to the measure $\mu$. Thus the measure $\mu$ is a fixed point of the mapping
		$\sigma\mapsto \Phi(\sigma)$, where $\Phi(\sigma)$ is the set of all projections of measures $\Pi\in M_R(\sigma)$ on $(x, t)$.
        Let $e_{x, t}(u, x, t)=(x, t)$. Denote by $\Pi\circ e_{x, t}^{-1}$ the projection of $\Pi$ on $(x, t)$.
        Note that the mapping $\Pi\mapsto \Pi\circ e_{x, t}^{-1}$ is linear and continuous with respect to the weak topology.
        It follows that the set $\Phi(\sigma)$ is nonempty, convex and compact in the weak topology.
        Recall that $\mathcal{M}_{R, M}^{\omega}(V)$ is convex (see Lemma~\ref{lem2}).

        The existence of a fixed point of $\Phi$ is based on the Kakutani--Ky Fan theorem
        (see, for instance, \cite{kakut}):  if a multivalued mapping $\Psi$
        from a convex compact set $\mathcal{C}$ in a locally convex space to the set of non-empty
        convex compact subsets of $\mathcal{C}$ has a closed graph, then there exists a point $p$ such
        that $p\in\Psi(p)$.

        Thus we need to verify that our mapping $\Phi$ has a closed graph.
        Assume that $\mu^n\in \Phi(\sigma^n)$, $\sigma^n\to\sigma$ and $\mu^n\to \mu$.
        Let $\mu^n=\Pi^n\circ e_{x, t}^{-1}$, where $\Pi^n\in M_R(\sigma^n)$. Since $\Pi^n\in P_R^{\omega}$
        and $P_R^{\omega}$ is a compact set (see the definition before Lemma~\ref{lem5}),
        there exists a subsequence $\Pi^{n_k}$ such that the measures $\Pi^{n_k}$
        converge weakly to a measure $\Pi\in P_R^{\omega}$. Note that $\mu=\Pi\circ e_{x, t}^{-1}$.
        In what follows we assume that the original sequence $\Pi^n$ converges weakly to $\Pi$.
        We need to prove that $\Pi\in M_R(\sigma)$.

        Applying Lemma~\ref{lem3}, we obtain that $\Pi\in S_{R}(\sigma)$.
        By Lemma~\ref{lem5} we have the inequality
		$$
		\lim\inf_{n\to\infty}\mathcal{F}_{\sigma_n}(\Pi^n)\ge \mathcal{F}_{\sigma}(\Pi).
		$$
		Let $\Gamma\in S_{R}(\sigma)$ and $\eta=\eta_t(dx)\,dt$ be a projection of $\Gamma$ on $(x, t)$.
		Denote by $\Gamma_{x, t}(du)$ the conditional measures for $\Gamma$.
        Let $(x, t)\mapsto v(x, t)$ be a Borel function such that the equality
		$$
		v(x, t)=\int_Uu\,\Gamma_{x, t}(du)
		$$
        holds for $\eta$ -- almost all $(x, t)$.
		Set
		$$
		\widetilde{\Gamma}(dudxdt)=\delta_{v(x, t)}(du)\eta_t(dx)\,dt.
		$$
		Arguing as in the part (II), we get $\widetilde{\Gamma}\in S_R(\sigma)$ and
		$\mathcal{F}_{\sigma}(\widetilde{\Gamma})\le \mathcal{F}_{\sigma}(\Gamma)$.
		
		By Lemma~\ref{lem4} there exists a sequence $\widetilde{\Gamma}^n\in S_{R}(\sigma^n)$ such that
		the measures $\widetilde{\Gamma}^n$ converge weakly to the measure $\widetilde{\Gamma}$
		and for every $n$ the projection of $\widetilde{\Gamma}^n$ on $(u, t)$
		is equal to the projection of $\widetilde{\Gamma}$ on $(u, t)$.
        By Lemma~\ref{lem5} we have the equality
		$$
		\lim_{n\to\infty}\mathcal{F}_{\sigma^n}(\widetilde{\Gamma}^n)=\mathcal{F}_{\sigma}(\widetilde{\Gamma}).
		$$
		Using the inequality
		$\mathcal{F}_{\sigma^n}(\Pi^n)\le \mathcal{F}_{\sigma_n}(\widetilde{\Gamma}^n)$,
		we obtain
		$\mathcal{F}_{\sigma}(\Pi)\le \mathcal{F}_{\sigma}(\widetilde{\Gamma})\le \mathcal{F}_{\sigma}(\Gamma)$.
		Thus $\Pi\in M_R(\sigma)$ and it follows that $\mu\in \Phi(\sigma)$. This completes the proof.
	\end{proof}
	
	Let us prove Corollary~\ref{cor1}.

	\begin{proof}
		Let us consider the function $u$ and the measure $\mu$ constructed in Theorem~\ref{th1}.
		The measure $\mu$ is a probability solution to the Cauchy problem $\partial_t\mu_t=L_{\mu, u(x, t)}^{*}\mu_t$, $\mu_0=\nu$.
		Moreover, we have $V\in L^1(\mu)$. According to the condition (H2.3), the coefficients of $L_{\mu, u(x, t)}$
        are integrable over $\mathbb{R}^d\times[0, T]$ with respect to the measure $\mu_t\,dt$.
		Applying the superposition principle (see, for instance \cite{Superp21}, \cite{Trev}) to the measure $\mu_t\,dt$,
        we obtain a probability measure $P$ on the space
		$C([0, T], \mathbb{R}^d)$ such that $\mu_t=P\circ e_t^{-1}$, where $e_t(\omega)=\omega(t)$, and for every function
		$\psi\in C_0^{\infty}(\mathbb{R}^d)$ the process
		$$
		\xi_t(\omega)=\psi(\omega(t))-\psi(\omega(0))-\int_0^tL_{\mu, u(s, \omega(s))}\psi(\omega(s), s)\,ds
		$$
		is a martingale with respect to the measure $P$ and the natural filtration $\mathcal{F}_t=\sigma(\omega(s), s\le t)$.
		By Proposition 2.1 from \cite[Chapter 4]{Ikeda-Watanabe} there exists a filtered probability
        space $(\Omega, \mathcal{F}_t, \mathcal{P})$ supporting
		a $\mathcal{F}_t$--Brownian motion $W$ and a $\mathcal{F}_t$--adapted process $X$
      	such that
		$$
		dX_t=\sqrt{2A(X_t, t, \mu)}dW_t+\bigl(b(X_t, t, \mu\bigr)+Q(X_t, t, \mu)u(X_t, t)\bigr)\,dt
		$$
		and $\mathcal{P}\circ X_t^{-1}=\mu_t$ for all $t\in[0, T]$. Thus we take the pair $(u, \mu)$ satisfying the assertions (i), (ii) and (iii) of Theorem~\ref{th1}, apply the superposition principle and obtain the processes~$X_t$ and $u(t, X_t)$ satisfying the assertions (i) and (ii) of Corollary~\ref{cor1}. Let us verify whether assertion~(iii) of Corollary~\ref{cor1} holds for $\mu_t$, $X_t$ and $u(t, X_t)$.

		Assume that $(\widetilde{\Omega}, \widetilde{\mathcal{F}}_t, \widetilde{\mathcal{P}})$ is another filtered probability space
        supporting a $\widetilde{\mathcal{F}}_t$--Brownian motion $\widetilde{W}$, a $\widetilde{\mathcal{F}}_t$--adapted process $Y$ and a $\widetilde{\mathcal{F}}_t$--adapted process $V$
      	such that $\nu=\widetilde{\mathcal{P}}\circ Y_0^{-1}$, $\mathbb{E}h(|V_t|)<\infty$ and
		$$
		dY_t=\sqrt{2A(Y_t, t, \mu)}d\widetilde{W}_t+\bigl(b(Y_t, t, \mu\bigr)+Q(Y_t, t, \mu)V_t\bigr)\,dt.
		$$
		Set $\sigma_t=\widetilde{\mathcal{P}}\circ Y_t^{-1}$. Denote by $\Pi_t(dvdy)$ the joint distribution of $(V_t, Y_t)$
        and $\Pi(dvdydt)=\Pi_t(dvdy)\,dt$.
        By the It$\hat{o}$ formula the equality
		\begin{multline*}
			\int_{\mathbb{R}^d}\psi(y)\,\sigma_t(dy)-\int_{\mathbb{R}^d}\psi(y)\,\nu(dy)=
			\\
			\int_0^t\int_{\mathbb{R}^d}L_{\mu, 0}\psi(y, s)\,\sigma_s(dy)\,ds+
			\int_{U\times\mathbb{R}^d\times [0, t]}\langle Q(y, s, \mu)v, \nabla\psi(y)\rangle\,\Pi(dvdyds).
		\end{multline*}
        holds for every $\psi\in C_0^{\infty}(\mathbb{R}^d)$ and all $t\in[0, T]$.
		Denote by $\Pi_{y, s}(dv)$ the conditional measures for $\Pi$ with respect to $\sigma_s\,ds$.
		There exists a Borel function $(y, s)\mapsto v(y, s)$ such that the equality
		$$
		v(y, s)=\int_{U}v\,\Pi_{y, s}(dv)
		$$
        holds for $\sigma_s\,ds$ -- almost all $(y, s)$.
		Since
		$$
		\int_{U\times\mathbb{R}^d\times[0, t]}\langle Q(y, s, \mu)v, \nabla\psi(y)\rangle\,\Pi(dvdyds)
		=\int_0^t\int_{\mathbb{R}^d}\langle Q(y, t, \mu)v(y, s), \nabla\psi(y)\rangle\sigma_s(dy)\,ds,
		$$
		the measure $\sigma_t\,dt$ is a probability solution to the Cauchy problem
		$\partial_t\sigma_t=L_{\mu, v(y, t)}^{*}\sigma_t$, $\sigma_0=\nu$.
		Note that
		$$
		\mathbb{E}\Bigl[\int_0^{T}f(V_t, Y_t, t, \mu)\,dt+g(Y_{T}, \mu)\Bigr]=
		\int_{U\times\mathbb{R}^d\times[0, T]}f(v, y, t, \mu)\Pi(dvdydt)\,dt+\int_{\mathbb{R}^d}g(y, \mu)\sigma_T(dy).
		$$
		Recall that $f$ is convex in $u$. Applying Jensen's inequality, we obtain
        $$
        \int_{U\times\mathbb{R}^d\times[0, T]}f(v, y, t, \mu)\,\Pi(dvdydt)\ge
		\int_0^T\int_{\mathbb{R}^d\times U}f(v(y, t), y, t, \mu)\sigma_t(dy)\,dt.
		$$
        According to the assertion (iii) of Theorem~\ref{th1}, the pair $(u, \mu)$ is optimal that is the following inequality
        \begin{multline*}
        \int_0^{T}\int_{\mathbb{R}^d}f(v(x, t), x, t, \mu)\,\sigma_t(dx)\,dt
			+\int_{\mathbb{R}^d}g(x, \mu)\sigma_{T}(dx)\ge			
\\
\int_0^{T}\int_{\mathbb{R}^d}f(u(x, t), x, t, \mu)\,\mu_t(dx)\,dt
			+\int_{\mathbb{R}^d}g(x, \mu)\mu_{T}(dx)
		\end{multline*}
		holds. In terms of the processes $X_t$, $u(X_t, t)$, $Y_t$ and $V_t$
        the last inequality can be rewritten in the following way
        $$
        \mathbb{E}\Bigl[\int_0^{T}f(V_t, Y_t, t, \mu)\,dt+g(Y_{T}, \mu)\Bigr]\ge
		\mathbb{E}\Bigl[\int_0^{T}f(u(X_t, t), X_t, t, \mu)\,dt+g(X_{T}, \mu)\Bigr].
		$$
        Thus the statement (iii) of Corollary~\ref{cor1} is proved.

	\end{proof}

\vspace*{0.2cm}

\centerline{\bf Acknowledgements.}

\vspace*{0.2cm}

The authors are grateful to Prof. Vladimir I. Bogachev, Prof. Yurii V. Averboukh and Tikhon~I.~Krasovitskii for fruitful discussions and valuable remarks. We thank to the anonymous referee for thorough
reading and useful corrections and suggestions.

This paper is supported by the Russian Science Foundation Grant N 25-11-00007 at the
Lomonosov Moscow State University.

\end{document}